\documentclass[a4paper,12pt,oneside]{article}
\usepackage[utf8]{inputenc}
\usepackage[english]{babel}
\usepackage{amsmath}
\usepackage{amsthm}
\usepackage{amssymb}
\usepackage{mathtools}
\usepackage{tikz}
\usepackage{verbatim}
\usepackage{fancyvrb}
\usepackage{array}
\usepackage{adjustbox}
\usepackage{changepage}
\usepackage{subcaption}
\usepackage{breqn}
\usepackage{enumerate}
\usepackage{wrapfig}
\usepackage{bbm}
\usepackage{bm}
\usepackage{enumitem}
\usepackage{csquotes}
\usepackage{hyperref}
\usepackage{framed}
\usepackage{enumitem}
\usepackage{dsfont}
\newtheorem{defi}{Definition}[section]
\newtheorem{thm}[defi]{Theorem}

\newtheorem{prop}[defi]{Proposition}

\newtheorem{rem}[defi]{Remark}

\newtheorem{assum}{Assumption}

\newcommand{\N}{\mathbb{N}}
\newcommand{\D}{\mathbb{D}}
\newcommand{\Z}{\mathbb{Z}}

\newcommand{\C}{\mathbb{C}}
\newcommand{\R}{\mathbb{R}}
\newcommand{\E}{\mathbb{E}}

\newcommand{\eps}{\epsilon}
\newcommand{\p}{\varphi}

\newcommand{\RP}[1]{\text{Re}\left(#1\right)}
\newcommand{\IP}[1]{\text{Im}\left(#1\right)}
\newcommand{\supp}[1]{\text{supp}(#1)}

\title{Limit theorems for Coulomb gases on a Jordan curve in an external potential}
\author{Kurt Johansson\footnote{KTH Royal Institute of Technology, Sweden. E-mail address: \texttt{kurtj[at]kth.se}}, Thomas Wolfs\footnote{KTH Royal Institute of Technology, Sweden. E-mail address: \texttt{wolfs[at]kth.se}.}}
\date{}

\begin{document}

\maketitle

\begin{abstract}
We consider a Coulomb gas on a Jordan curve $\gamma$ in an external potential $V$ at inverse temperature $\beta>0$ and obtain an asymptotic expansion of the free energy up to $o(1)$ and a central limit theorem for linear statistics. We focus on the one-cut regime, where the density of the weighted equilibrium measure of $\gamma$ in $V$ is strictly positive on $\gamma$. The constant term in the (normalized) expansion consists of two parts: the Fredholm determinant of a generalized Grunsky operator and  the Dirichlet energy of the logarithm of the density of the weighted equilibrium measure of $\gamma$. The coefficient of the latter vanishes for $\beta=2$. The variance of the fluctuations of the linear statistics only depends on the Dirichlet energy of the test function and is therefore independent of $V$. Essential in our approach is that the generalized Grunsky operator and the accompanying equilibrium parametrization allow us to transport the particles on the curve in the external potential to a reference object in a way that preserves the equilibrium measure. In our setting, the unit circle is the natural reference object. 
\end{abstract}

\section{Introduction}
Let $\gamma$ be a rectifiable Jordan curve in the complex plane with $0$ in its interior. We will consider a Coulomb gas constrained to $\gamma$ in an external potential $V : \gamma \to \mathbb{R}$ at inverse temperature $\beta>0$. As a statistical mechanical system, this is the Gibbs measure 
\begin{equation} \label{CG_def}
\frac{1}{n! Z_{n,\beta}^{\gamma,V}}
\exp\left(
-\frac{\beta}{2} \mathcal{H}_n^{V}(z_1,\dots,z_n)
\right)
\prod_{k=1}^n |dz_k|,
\end{equation}
on $\gamma^n$ with Hamiltonian
\begin{equation} \label{H_n_def}
    \mathcal{H}_n^{V}(z_1,\dots,z_n) = \sum_{1\leq k\neq l\leq n}
\log \frac{1}{|z_k - z_l|}
+ n \sum_{k=1}^n V(z_k),
\end{equation}
The normalizing constant $Z_{n,\beta}^{\gamma,V}>0$ is known as the partition function of the system and is given by
\begin{equation}\label{ZnV}
Z_{n,\beta}^{\gamma,V}
=
\frac{1}{n!}
\int_{\gamma^n}
\exp\left(
-\frac{\beta}2 \mathcal{H}_n^{V}(z_1,\dots,z_n)
\right)
\prod_{k=1}^n |dz_k|.
\end{equation}

Essential in our consideration will be the notion of equilibrium parametrization, which is intimately related to objects from potential theory, see \cite{SaffTotik1997} for an introduction. Recall that the \textit{weighted equilibrium measure} $\mu_{\gamma,V}$ on $\gamma$ in a continuous external potential $V:\gamma\to\R$ is the unique minimizer among all Borel probability measures on $\gamma$ of the \textit{weighted logarithmic energy}
\[
I_{\gamma,V}[\mu]
=
\iint_{\gamma^2}
\log \frac{1}{|u-v|}
d\mu(u) d\mu(v)
+
\int_\gamma V(u) d\mu(u), 
\]
It is an elementary fact that the equilibrium measure $d\mu_{\mathbb{T},0}(z)$ on the unit circle $\mathbb{T}$ in the absence of an external potential, is given by the arc-length measure $|dz|/2\pi$. Throughout this paper, we will always impose the following assumption on $V$.

\begin{assum} \label{ass1} The potential $V:\gamma\to \R$ is continuous and such that $d\mu_{\gamma,V}(z)=w_{\gamma,V}(z) |dz|$ with $w_{\gamma,V}>0$ on $\gamma$ (known as the one-cut regime).
\end{assum}

An important property of the exterior conformal map $\phi$, which maps the exterior $\D^*$ of the unit disc $\D=\{z\in\C:|z|\,<1\}$ to the outer domain $D^*$ of $\gamma$, is that it transports the equilibrium measure on $\mathbb{T}$ to the equilibrium measure on $\gamma$. Due to the inclusion of the external potential in \eqref{CG_def}, in contrast to \cite{Johansson2022,CourteautJohansson2025} where $V=0$, we won't be able use this transport property anymore to analyze the properties of the Coulomb gas. The idea of using a different kind a transport map first appeared in \cite{CourteautJohanssonViklund2026} to study Coulomb gases on a Jordan arc and led to the introduction of the arc-Grunsky operator. The latter can be seen as the Jordan arc analogue of the classical Grunsky operator, see \cite{Pommerenke1975} for an introduction. Our approach will further build on this and will introduce a more general equilibrium parametrization and Grunsky operator, allowing us to transfer the equilibrium measure of a wide variety of reference objects to objects in an external potential. For a sufficiently regular Jordan curve in a one-cut potential, the unit circle $\mathbb{T}$ will be a natural reference object.
\medbreak

Our main interest lies in describing an asymptotic expansion of the free energy $\log Z_{n,\beta}^{\gamma,V}$ as $n\to\infty$. This is motivated by the fact that the coefficients in this expansion encode information about the underlying physical model as predicted by \cite{WiegmannZabrodin2022,WiegmannZabrodin2006} in the physics literature. Particular interest goes to the constant term in the expansion, which is believed to be intimately related to conformal field theory, and is typically the most challenging to describe rigorously. We will establish the free energy expansion up to $o(1)$ for Coulomb gases on a Jordan curve in a general, sufficiently regular, one-cut potential at arbitrary $\beta>0$. The constant term in the (normalized) expansion will consist of two parts: the Fredholm determinant of a generalized Grunsky operator and the Dirichlet energy of the logarithm of the density of the weighted equilibrium measure on $\gamma$. The coefficient of the latter will vanish for $\beta=2$, and for $V=0$, the remaining part will be a scalar multiple of the Loewner energy of the curve, see \cite{Wang19,JohanssonViklund26}.
\medbreak

A major step into describing the free energy expansion of a Coulomb gas on a Jordan curve was taken in \cite{Johansson2022}, where the case $V=0$ and $\beta=2$ was treated, and the constant term was described in terms of a Fredholm determinant involving the classical Grunsky operator. An analogous expansion at any $\beta>0$ was proven later in \cite{CourteautJohansson2025}. Another work in this direction is \cite{BlackstoneCharlierLenells2024}, which, using Riemann-Hilbert methods for Toeplitz determinants, treats the case where $\gamma=\mathbb{T}$, the potential is one-cut regular and $\beta=2$.

Asymptotic expansions of the free energy have also been established for planar Coulomb gases rather than Coulomb gases on a curve. In that setting, the Jordan curve $\gamma$ and arc-length measures $|dz_k|$ are replaced by the complex plane $\C$ and the area measures $dA(z_k)$ respectively. In the seminal work \cite{LebleSerfaty2017}, an asymptotic expansion of the free energy up to $o(n)$ was obtained for planar Coulomb gases in a general, sufficiently regular, confining potential at arbitrary $\beta > 0$, see also \cite{Bauerschmidt_etal}. More recently, at $\beta=2$, expansions up to $o(1)$ have been established for planar Coulomb gases in a general, sufficiently regular, radially symmetric potential, see~\cite{ByunKangSeo2023,AmeurCharlierCronvall2026}, as well as in certain special Ginibre-type non-radially symmetric potentials, see~\cite{ByunSeoYang2025,Byun2025,ByunYangYoo2026}. To the best of our knowledge, an expansion up to $o(1)$ involving a general class of non-radially symmetric potentials at arbitrary $\beta>0$ remains open. As described above, we will provide such an expansion in the setting of a Coulomb gas on a Jordan curve.
\medbreak

Our second interest lies in a central limit theorem for linear statistics of the particles in the Coulomb gas \eqref{CG_def}. We will establish a central limit theorem for Coulomb gases on a Jordan curve in a general, sufficiently regular, one-cut potential at arbitrary $\beta>0$. The variance will only depend on the Dirichlet energy of the test function and will therefore be independent of $V$.
\medbreak

Let $\E_{n,\beta}^{\gamma, V}[\,\cdot\,]$ denote expectation with respect to the Coulomb gas in \eqref{CG_def} and let $g:\gamma\to \R$ be a sufficiently regular test function. In the case $\gamma=\mathbb{T}$, $V=0$ and $\beta=2$, the central limit theorem follows from the classical strong Szeg\H{o} limit theorem for Toeplitz determinants, see \cite{Johansson88,Simon}, which implies that 
\begin{equation}\label{ThmSz_0}
    \lim_{n\to\infty}\log \E_{n,2}^{\mathbb{T}, 0}\left[e^{\sum_{k=1}^ng(z_k)-n\int_{\mathbb{T}}g(z)\frac{|dz|}{2\pi}}\right]= \frac{1}{8\pi} \int_{\C\setminus\gamma} |\nabla g_{\rm ext}|^2 \, dxdy,
\end{equation}
where $g_{\rm ext}$ denotes the extension of $g$ that is bounded and harmonic on the connected components of $\C\setminus\gamma$. The analogous result at arbitrary $\beta>0$ was established in \cite{Webb2016}. Mesoscopic and high temperature variants of \eqref{ThmSz_0} were obtained in \cite{Lambert2019} and \cite{HardyLambert2021} respectively. Generalizations to sufficiently regular Jordan curves $\gamma$ and $V=0$ were described in \cite{Johansson88,Johansson2022} at $\beta=2$ and in \cite{CourteautJohansson2025} at arbitrary $\beta>0$. The case where $\gamma=\mathbb{T}$, the potential is one-cut regular and $\beta=2$ is covered by \cite{BlackstoneCharlierLenells2024}.

Central limit theorems for planar Coulomb gases in general, sufficiently regular, confining potentials were obtained in \cite{AmeurHedenmalmMakarov2011} at $\beta=2$ and in \cite{LebleSerfaty2018,Bauerschmidt_etal} at arbitrary $\beta>0$.

\section{Main results}

For the convenience of the reader, we will first give an overview of some of the results in \cite{Johansson2022,CourteautJohansson2025} where $V=0$. Recall that the exterior conformal map $\phi:\D^*\to D^*$ has the expansion
\[
\phi(z)=\text{cap}(\gamma)z+\phi_0+\phi_{-1}z^{-1}+\dots,\quad z\to\infty,
\]
where $\text{cap}(\gamma)>0$ is the (ordinary) logarithmic capacity of $\gamma$. If $u,v\in\mathbb{D}^\ast$, there exists $a_{k,l}\in\C$, known as the \textit{Grunsky coefficients}, such that
\begin{equation}\label{ClGr}
\log\frac{\phi(u)-\phi(v)}{u-v}=\log(\text{cap}(\gamma))-\sum_{k, l\ge 1}a_{k l}u^{-k}v^{- l}.
\end{equation}
This identity is often referred to as the Grunsky expansion. The classical \textit{Grunsky operator} $B: \ell^2(\Z_{\geq 1})\to \ell^2(\Z_{\geq 1})$ is defined as $B=(b_{kl})_{k,l\geq 1}$ where $b_{kl}=\sqrt{kl}a_{kl}$, see \cite{Pommerenke1975}. It was proven in \cite{Johansson2022} that if $B$ is a Hilbert-Schmidt operator, then
\begin{equation}\label{Vzero}
\lim_{n\to\infty}\log\frac{Z_{n,2}^{\gamma,0}}{\text{cap}(\gamma)^{n^2}(2\pi)^n}=-\frac 12\log\det(I-BB^*).
\end{equation}
As a consequence, the free energy admits the following asymptotic expansion
\begin{equation}\label{Vzero_FE}
     \log Z_{n,2}^{\gamma,0} = n^2 \log \text{cap}(\gamma) + n\log (2\pi) - \frac 12\log\det(I-BB^*) + o(1),\quad n\to\infty.
\end{equation}
The quantity on the right hand side of \eqref{Vzero} is, up to a multiplicative constant, the Loewner energy of the Jordan curve $\gamma$, see \cite{Wang19,JohanssonViklund26}. An analogous limit to \eqref{Vzero} at any $\beta>0$ was described in \cite{CourteautJohansson2025}. There, a real version of $B$ appears, which arises by taking the real part of \eqref{ClGr}. Doing so, we obtain
\begin{equation}\label{ReGr}
\log\left|\frac{u-v}{\phi(u)-\phi(v)}\right|=-\log(\text{cap}(\gamma))+
\Psi(u)^t K \Psi(v)
, \quad u,v\in \mathbb{T},
\end{equation}
in terms of the basis
\[
\Psi(e^{it})=
\begin{pmatrix} (\frac 1{\sqrt{k}}\cos(kt))_{k\ge 1} \\
(\frac 1{\sqrt{k}}\sin(kt))_{k\ge 1}\end{pmatrix},\quad t\in[0,2\pi],
\]
and an operator $K$, which is related to the classical Grunsky operator via
\begin{equation}\label{clK}
   K=\begin{pmatrix} {\rm{Re}}\,B & {\rm{Im}}\,B  \\ {\rm{Im}}\,B & -{\rm{Re}}\,B \end{pmatrix}.
\end{equation}
The operator $K$ is a symmetric operator on $ \ell^2(\Z_{\geq 1})\oplus  \ell^2(\Z_{\geq 1})$ and satisfies
$$\det(I+K)=\det(I-BB^*).$$
There is also a natural connection to the classical Neumann-Poincaré operator, see \cite[Prop. 2.7]{CourteautJohansson2025}.
\medbreak

The \textit{equilibrium parametrization} $z_e:\mathbb{T}\to\gamma$ of the Jordan curve $\gamma$ in the external potential $V$ by $\mathbb{T}$ will be defined through
\[
\mu_{\gamma,V}([z_e(1),z_e(e^{it})])=\frac {t}{2\pi},\quad t\in[0,2\pi],
\]
where $[z_e(1),z_e(e^{it})]$ is the positively oriented arc on $\gamma$ from $z_e(1)$ to $z_e(e^{it})$. We will show that, in analogy with \eqref{ReGr}, the accompanying kernel
\begin{equation}\label{genGr}
    \mathcal{G}(u,v) = \log\left|\frac{u-v}{z_e(u)-z_e(v)}\right|+\frac 12(V(z_e(u))+V(z_e(v))),\quad u,v\in\mathbb{T},
\end{equation}
admits an expansion of the form
\begin{equation}\label{weK}
\mathcal{G}(u,v)=-\log(\text{cap}_V(\gamma))+ 
\Psi(u)^t K_{\mathbb{T}}^{\gamma,V}
\Psi(v),\quad u,v\in\mathbb{T},
\end{equation}
in terms of an operator $K_{\mathbb{T}}^{\gamma,V}$ on $ \ell^2(\Z_{\geq 1})\oplus  \ell^2(\Z_{\geq 1})$, which we will call the \textit{weighted Grunsky operator}. Recall that $\text{cap}_V(\gamma)$ denotes the \textit{weighted logarithmic capacity} of $\gamma$, which is given by
\[
\text{cap}_V(\gamma)=e^{-I_{\gamma,V}[\mu_{\gamma,V}]}.
\]
For $V=0$, it follows from \cite[Thm. 4.3.4 \& 4.3.14]{Ransford1995} that $z_e=\left.\phi\right|_{\mathbb{T}}$, so that the operator $K_{\mathbb{T}}^{\gamma,0}$ is precisely the operator $K$ in \eqref{clK}. 
\medbreak

In Section \ref{eq_par}, we will study the notions of equilibrium parametrization and generalized Grunsky operator in a somewhat general framework, starting from a general contour in an external potential and reference contour, together with some compatibility conditions on their equilibrium measures. Due to the inclusion of the external potential in \eqref{genGr}, properties of the classical Grunsky operator, which are typically proven via geometric function theory, don't necessarily carry over to the generalized Grunsky operator. Interestingly, we will be able to show that many of the important properties of the classical Grunsky operator do still carry over, purely based on arguments from potential theory.  Importantly, we will show that there is still a one-sided strengthened Grunsky inequality, see \cite[\S 9.4]{Pommerenke1975} for the two-sided version in the classical setting. It will follow from these properties that, under suitable regularity assumptions on $\gamma$ and $V$, the Fredholm determinant $\det(I+K_{\mathbb{T}}^{\gamma,V})$ of the generalized Grunsky operator $K_{\mathbb{T}}^{\gamma,V}$ in \eqref{weK} is a well-defined strictly positive quantity.
\medbreak

Before stating our main result, we have to introduce some additional notation. Recall that the Dirichlet inner product on $\C\setminus\gamma$ of two functions $f,g:\gamma\to\C$ is defined as 
\begin{equation} \label{Dir_IP}
    \mathcal{D}_{\C\setminus\gamma}(f,g)= \frac{1}{4\pi} \iint_{\C\setminus\gamma} \nabla f_{\rm ext} \cdot \overline{\nabla g_{\rm ext}} \, dxdy,
\end{equation}
where $f_{\rm ext}$ and $g_{\rm ext}$ denote the extensions of $f$ and $g$ that are bounded and harmonic on the connected components of $\C\setminus\gamma$. The Dirichlet energy of $f$ is then given by
    $$ \mathcal{D}_{\C\setminus\gamma}(f) = \mathcal{D}_{\C\setminus\gamma}(f,f), $$
so that
    $$ \mathcal{D}_{\C\setminus\gamma}(f)= \frac{1}{4\pi} \int_{\C\setminus\gamma} |\nabla f_{\rm ext}|^2 \, dxdy. $$
Define the \textit{normalized partition function} by
$$\bar{Z}_{n,\beta}^{\gamma,V}
= Z_{n,\beta}^{\gamma,V} \exp\left(\frac{\beta}{2}n^2 I_{\gamma,V}[\mu_{\gamma,V}] + \bigg(1-\frac{\beta}{2}\bigg)n\int_{\gamma}\log w_{\gamma,V}d\mu_{\gamma,V}\right).
$$
Note that $\bar{Z}_{n,\beta}^{\mathbb{T},0}=(2\pi)^{(\frac{\beta}{2}-1)n} Z_{n,\beta}^{\mathbb{T},0}$ and that $Z_{n,\beta}^{\mathbb{T},0}$ is given by a Selberg integral, see \cite[Thm. 12.1.1]{Mehta2004}, which evaluates to
    $$Z_{n,\beta}^{\mathbb{T},0} = \frac{(2\pi)^n}{n!} \frac{\Gamma(1+\frac{\beta n}{2})}{\Gamma(1+\frac{\beta}{2})^n}. $$
Our main result is now as follows. We will prove this in Section \ref{proofs}.

\begin{thm} \label{PF}
Let $\gamma$ be a $C^{11,\alpha}$ Jordan curve and assume that $V\in C^{11,\alpha}(\gamma)$ satisfies Assumption \ref{ass1}. Let $K_{\mathbb{T}}^{\gamma,V}$ be the weighted Grunsky operator on $ \ell^2(\Z_{\geq 1})\oplus  \ell^2(\Z_{\geq 1})$ defined by \eqref{weK}. Then,
    $$\lim_{n\to\infty}\log \frac{\bar{Z}_{n,\beta}^{\gamma,V}}{\bar{Z}_{n,\beta}^{\mathbb{T},0}} = -\frac{1}{2} \log\det(I+K_{\mathbb{T}}^{\gamma,V}) + \frac{1}{2} \bigg(\sqrt{\frac{2}{\beta}}-\sqrt{\frac{\beta}2}\bigg)^2\mathcal{D}_{\C\setminus\gamma} (\log w_{\gamma,V}).$$
\end{thm}

A consequence of Theorem \ref{PF} is that we obtain an asymptotic expansion of the free energy up to $o(1)$ of the Coulomb gas in \eqref{CG_def}, namely
\begin{align*}
    \log Z_{n,\beta}^{\gamma,V} = &-\frac{\beta}{2} n^2 I_{\gamma,V}[\mu_{\gamma,V}] + \bigg(\frac{\beta}{2}-1\bigg) n \left( \int_{\gamma} \log w_{\gamma,V}d\mu_{\gamma,V}-\log \frac{1}{2\pi}\right) + \log Z_{n,\beta}^{\mathbb{T},0} \\
    &- \frac{1}{2} \log\det(I+K_{\mathbb{T}}^{\gamma,V}) + \frac{1}{2}\bigg(\sqrt{\frac 2{\beta}}-\sqrt{\frac{\beta}2}\bigg)^2\mathcal{D}_{\C\setminus\gamma} (\log w_{\gamma,V}) + o(1),\quad n\to\infty.
\end{align*}
The appearance of the entropy-like integral $\int_{\gamma}\log w_{\gamma,V}\,d\mu_{\gamma,V}$ at order $n$ is analogous to what was observed for planar Coulomb gases in \cite{LebleSerfaty2017}.
\medbreak

\medbreak

Specializing to $\beta=2$ then yields the following result.

\begin{thm}\label {Thmbeta2}
Under the same assumptions as in Theorem \ref{PF}, we have that
\[
\lim_{n\to\infty}\log\frac{Z_{n,2}^{\gamma,V}}{{\rm cap}_V(\gamma)^{n^2}(2\pi)^n}=-\frac 12\log\det(I+K_{\mathbb{T}}^{\gamma,V}).
\] 
\end{thm}
For $V=0$, by our earlier discussion, this recovers \eqref{Vzero} and establishes the connection to the Loewner energy of $\gamma$. For $\gamma=\mathbb{T}$ and real analytic $V$, this recovers \cite[Thm. 1.1]{BlackstoneCharlierLenells2024} with $W=0$ and all $\alpha_k,\beta_k=0$. 
\medbreak

The following important result will also be proven in Section \ref{proofs}.

\begin{thm}\label{ThmSz}
Let $\gamma$ be a $C^{7,\alpha}$ Jordan curve and assume that $V\in C^{7,\alpha}(\gamma)$ satisfies Assumption~\ref{ass1}. If $g\in C^{4,\alpha}(\gamma)$, then
\[
\lim_{n\to\infty} \log \E_{n,\beta}^{\gamma, V}\left[e^{\sum_{k=1}^ng(z_k)-n\int_{\gamma}g d\mu_{\gamma,V}}\right]=(1-\frac{2}{\beta}) \mathcal{D}_{\C\setminus\gamma}(g,\log w_{\gamma,V})+\frac 1{\beta}\mathcal{D}_{\C\setminus\gamma}(g).
\]
\end{thm}
A consequence of Theorem \ref{ThmSz} is that we have a central limit theorem for the linear statistic $\sum_{k=1}^ng(z_k)$ of the particles in the Coulomb gas, namely
$$\sum_{k=1}^n g(z_k) - n \int g d\mu_{\gamma,V} \stackrel{\rm distr}{\longrightarrow} \mathcal{N}(m_g,\sigma_g^2), $$
where the mean $m_g$ and variance $\sigma_g^2$ are given by
\[
m_g = (1-\frac{2}{\beta}) \mathcal{D}_{\C\setminus\gamma}(g,\log w_{\gamma,V}),\quad \sigma_g^2 = \frac{2}{\beta} \mathcal{D}_{\C\setminus\gamma}(g).
\]
Note that the variance $\sigma_g^2$ does not depend on the external potential $V$ and only depends on the support $\gamma$ of the equilibrium measure. This phenomenon has also been observed for  Coulomb gases on the real line in \cite{Johansson1998,BekermanLebleSerfaty2018} and for planar Coulomb gases in \cite{LebleSerfaty2018}. It would be interesting to have a heuristic explanation of this fact.
\medbreak

For $\gamma=\mathbb{T}$, $V=0$ and $\beta=2$, Theorem \ref{ThmSz} recovers \eqref{ThmSz_0}. After making the proper identifications, for a $C^{9,\alpha}$ Jordan curve $\gamma$ and $V=0$, it recovers \cite[Cor. 1.4]{CourteautJohansson2025}, while for $\gamma=\mathbb{T}$, real analytic $V$ and $\beta=2$, it recovers \cite[Cor. 1.4]{BlackstoneCharlierLenells2024}.  Compared to \cite[Cor. 1.4]{CourteautJohansson2025}, we have slightly improved the regularity from a $C^{9,\alpha}$ Jordan curve to a $C^{7,\alpha}$ Jordan curve. This is due to the fact that we manage to compare the $\log$-kernel and the kernel in \eqref{weK} directly, rather than going through its derivatives.
\medbreak

The regularity assumptions in Theorem \ref{PF} \& \ref{ThmSz} are present for technical convenience in the proofs and are not optimal. Finding and proving the optimal conditions on the curve $\gamma$ and function $g$ is presumably a rather challenging problem. If $\gamma=\mathbb{T}$, $V=0$ and $\beta=2$, then the optimal condition on $g$ in Theorem \ref{ThmSz} is that it belongs to the Sobolev space $H^{1/2}(\mathbb{T})$, see \cite{Simon2005}, and if $\beta=2$ and $V=0$, then the optimal condition in Theorem~\ref{PF} is that $\gamma$ is a Weil-Petersson quasicircle, see \cite{Johansson2022} and \cite{Bishop25} for many characterizations of Weil-Petersson quasicircles.

\section{Equilibrium parametrization} \label{eq_par}

Given appropriate contours $\gamma$ and $\tau$ in the complex plane and a continuous external potential $V:\gamma\to\R$, we will construct a map $z_e:\tau\to\gamma$, which we will call the equilibrium parametrization of $\gamma$ in $V$ by $\tau$, with the transport property
\begin{equation}\label{EP_COV}
    \int (fz_e) d\mu_\tau = \int f d\mu_{\gamma,V},
\end{equation}
for every bounded Borel function $f$ on $\gamma$. In order to show existence and some basic properties, we will assume the following.

\begin{assum} \label{ass2}
    The contours $\gamma$ and $\tau$ admit arc-length parametrizations and satisfy $d\mu_{\gamma,V}(z)=w_{\gamma,V}(z) |dz|$ with $w_{\gamma,V}>0$ on $\gamma$ and $d\mu_{\tau}(z)=w_{\tau}(z) |dz|$ with $w_{\tau}>0$ on $\tau$.
\end{assum}

Under some additional assumptions, we will show that there is a natural operator accompanying $z_e$, which we will call the generalized Grunsky operator. We will show that many of the important properties of the classical Grunsky operator carry over to this setting as well. As an example, in the last subsection, we will restrict to the setting where $\gamma$ is a Jordan curve and $\tau=\mathbb{T}$. This will be relevant in the proof of Theorem \ref{PF} \& \ref{ThmSz}.

\subsection{Existence and regularity}

Slightly abusing notation, let $\tau:[a_\tau,b_\tau]\to \tau$ and $\gamma:[a_\gamma,b_\gamma]\to \gamma$ denote the arc-length parametrization of the contours $\tau$ and $\gamma$ respectively (normalized such that $|\tau'(t)|\,=1$ for all $t\in [a_\tau,b_\tau]$ and $|\gamma'(t)|\,=1$ for all $t\in [a_\gamma,b_\gamma]$). Consider the cumulative distribution functions 
\begin{align*}
    F_{\gamma,V}(t) &= \int_{a_\gamma}^{t} (w_{\gamma,V}\gamma)(u) du,\quad t\in[a_\gamma,b_\gamma], \\
    F_{\tau}(t) &= \int_{a_\tau}^{t} (w_{\tau}\tau)(u) du,\quad t\in [a_\tau,b_\tau].
\end{align*}
Note that to shorten the notation, we wrote $(fg)$ for the composition $f\circ g$. To improve the readability, this notation will appear throughout the rest of the manuscript as well. Since $\supp{w_{\gamma,V}}=\gamma$ and $\supp{w_{\tau}}=\tau$, the maps $F_{\gamma,V}:[a_\gamma,b_\gamma]\to [0,1]$ and $F_{\tau}:[a_\tau,b_\tau]\to [0,1]$ are bijective and absolutely continuous.
\medbreak
    
We define the equilibrium parametrization of $\gamma$ in the external field $V$ by $\tau$ as the map $z_e:\tau\to\gamma$ with 
\begin{equation} \label{eq_par_def}
        z_e(u) = (\gamma F_{\gamma,V}^{-1}F_\tau\tau^{-1})(u),\quad u\in\tau.
\end{equation}
Clearly, $(z_e\tau)(a_\tau)=\gamma(a_\gamma)$ and $(z_e\tau)(b_\tau)=\gamma(b_\gamma)$.
\medbreak

The following result is an immediate consequence of the definition.

\begin{prop} \label{EP_AC}
    Suppose that Assumption \ref{ass2} holds. The map $z_e:\tau\to\gamma$ in \eqref{eq_par_def} is bijective and absolutely continuous.
\end{prop}

Since $z_e\in{\rm AC}(\tau)$, its derivative $(z_e\tau)'$ exists almost everywhere on $[a_\tau,b_\tau]$. We can show that it is strictly positive almost everywhere.

\begin{prop}
    Under Assumption \ref{ass2}, the map $z_e:\tau\to\gamma$ in \eqref{eq_par_def} satisfies $|(z_e\tau)'(t)|\,>0$ for a.e. $t\in[a_\tau,b_\tau]$.
\end{prop}
\begin{proof}
    It follows from the definition that
    \begin{equation}\label{z_e_der}
        (z_e\tau)'(t) = (\gamma F_{\gamma,V}^{-1}F_\tau)'(t) = \gamma' ((F_{\gamma,V}^{-1}F_\tau)(t)) \frac{F_\tau'(t)}{ F_{\gamma,V}'((F_{\gamma,V}^{-1}F_{\tau})(t))},\quad \text{a.e. } t\in [a_\tau,b_\tau].
    \end{equation}
    It then remains to use the fact that $|\gamma'(t)|\ =1$, $F_\tau'(t) = (w_{\tau}\tau)(t)>0$ and $F_{\gamma,V}'(t) = (w_{\gamma,V}\gamma)(t)>0$ for all $t\in [a_\tau,b_\tau]$.
\end{proof}

The previous results imply that the kernel
$$\mathcal{G}(u,v) = \log\left|\frac{u-v}{z_e(u)-z_e(v)}\right| + \frac{1}{2} ((Vz_e)(u)+(Vz_e)(v)),$$
which will ultimately lead to the definition of the generalized Grunsky operator, is well-defined on $\tau\times\tau \setminus \{(u,u)\in \tau\times\tau : u\in\tau\}$ and can be extended to the diagonal $\{(u,u)\in \tau\times\tau : u\in\tau\}$ as
$$\mathcal{G}(u,u) = - \log\left|(z_e\tau)'(u)\right| + (Vz_e)(u),\quad \text{a.e. } u\in\tau.$$

Next, we show that the equilibrium parametrization satisfies $\eqref{EP_COV}$. This means that $\mu_{\gamma,V}$ is the pushforward measure of $\mu_\tau$ by $z_e$, often denoted by $\mu_{\gamma,V} = z_e\#\mu_\tau$.

\begin{prop}
    Suppose that Assumption \ref{ass2} holds. The map $z_e:\tau\to\gamma$ in \eqref{eq_par_def} is such that $\mu_{\gamma,V} = z_e\#\mu_\tau$.
\end{prop}
\begin{proof}
    Let $f$ be a bounded Borel function on $\gamma$. By definition,
    $$ \int f d\mu_{\gamma,V} = \int_{a_\gamma}^{b_\gamma} (f\gamma)(t) (w_{\gamma,V}\gamma)(t) dt, $$
    hence after the change of variables $t\mapsto \gamma^{-1}z_e \tau(t)=F_{\gamma,V}^{-1}F_{\tau}(t)$, we obtain
    $$\int_{a_\tau}^{b_\tau} (fz_e\tau)(t) (w_{\gamma,V}z_e\tau)(t) (F_{\gamma,V}^{-1}F_{\tau})'(t) dt .$$
    Now use the fact that 
    \begin{equation}
        (F_{\gamma,V}^{-1}F_{\tau})'(t) = \frac{F_\tau'(t)}{F_{\gamma,V}'((F_{\gamma,V}^{-1}F_{\tau})(t))} = \frac{(w_{\tau}\tau)(t)}{(w_{\gamma,V}z_e\tau)(t)},\quad \text{a.e. } t\in [a_\tau,b_\tau].
    \end{equation}
    to obtain  
        $$ \int f d\mu_{\gamma,V} = \int_{a_\tau}^{b_\tau} (fz_e\tau)(t) (w_{\tau}\tau)(t) dt = \int (fz_e) d\mu_\tau.$$
    This proves that $\mu_{\gamma,V} = z_e\#\mu_\tau$.
\end{proof}

For $C^{1,\alpha}$ Jordan curves $\gamma$ and $\tau$, and $V=0$, it follows from \cite[Thm. 4.3.4 \& 4.3.14]{Ransford1995} that the conformal map between the outer domain of $\tau$ and the outer domain of $\gamma$ is a map that satisfies \eqref{EP_COV}. As such we can see the equilibrium parameterization as a generalization of the conformal map restricted to the curve in the sense that it still preserves the equilibrium measure in the presence of an external field. For a Jordan arc $\gamma$, $V=0$ and $\tau=[-1,1]$, the definition in \eqref{eq_par_def} essentially reduces to the definition of the equilibrium parametrization in \cite[Eq. (1.1)]{CourteautJohanssonViklund2026}.
\medbreak

By picking an appropriate reference contour, initial regularity of the contour and the potential can be translated into regularity of the equilibrium parametrization. For example, we can show the following result in the setting of a Jordan curve.

\begin{prop} \label{ep_reg}
    Suppose that Assumption \ref{ass2} holds. If $\gamma$ and $\tau$ are $C^{m,\alpha}$ Jordan curves and $V\in C^{m,\alpha}(\gamma)$ with $m\geq 1$, then $z_e\in C^{m,\alpha}(\tau)$. Moreover, there exists $c>0$ such that $|(z_e\tau)'(t)|\, \geq c$ for a.e. $t\in[a_\tau,b_\tau]$.
\end{prop}
\begin{proof}
    Let $V_{\text{ext}}$ and $V_{\text{int}}$ be the bounded harmonic extension of $V:\gamma\to\R$ to $\text{ext}(\C\setminus\gamma)$ and $\text{int}(\C\setminus\gamma)$ respectively. It follows from the assumptions on $\gamma$ and $V$ and Kellogg's theorem that $V_{\text{ext}}$ and $V_{\text{int}}$ are in $C^{m,\alpha}$. From \cite[Thm. 4.7]{SaffTotik1997}, we know that
    \begin{align*}
        U_{\mu_{\gamma,V}}(z) &= I_{\mu_{\gamma,V}}[\mu_{\gamma,V}] - \frac{1}{2} \int V d\mu_{\gamma,V}  - \frac{1}{2} V_{\text{ext}}(z) - g_{\text{ext}(\C\setminus\gamma)}(z,\infty),\quad z\in \text{ext}(\C\setminus\gamma). \\
        U_{\mu_{\gamma,V}}(z) &= I_{\mu_{\gamma,V}}[\mu_{\gamma,V}] - \frac{1}{2} \int V d\mu_{\gamma,V} - \frac{1}{2} V_{\text{int}}(z),\quad z\in \text{int}(\C\setminus\gamma),
    \end{align*}
    where $ g_{\text{ext}(\C\setminus\gamma)}(z,\infty) = \log |\psi(z)| $ and $\psi:\text{ext}(\C\setminus\gamma)\to\text{ext}(\C\setminus\mathbb{T})$ is the conformal map with $\psi(\infty)=\infty$ and $\lim_{z\to\infty} \psi(z)/z>0$. The assumptions on $\gamma$ and Kellogg's theorem again imply that $\psi$ is in $C^{m,\alpha}$. Denote
        $$ U_{\mu_{\gamma,V}}(z) = \int \log \frac{1}{|z-v|} d\mu_{\gamma,V}(v) ,\quad z\in\C\setminus\gamma, $$
    and let $\mathfrak{n}_{+}$ and $\mathfrak{n}_{-}$ be the normal to $\gamma$ pointing into the exterior and interior so $\mathfrak{n}_{\pm} = \mp i \mathfrak{t}$ in terms of the unit tangent $\mathfrak{t}$ to $\gamma$. It follows from \cite[Thm. 1.5]{SaffTotik1997} that whenever $U_{\mu_{\gamma,V}}$ is Lipschitz in a neighborhood of $\supp{\mu_{\gamma,V}}=\gamma$, 
    $$ d\mu_{\gamma,V}(z) = - \frac{1}{2\pi} \left( (\partial_{\mathfrak{n}_+}U_{\mu_{\gamma,V}})(z) + (\partial_{\mathfrak{n}_-}U_{\mu_{\gamma,V}})(z) \right) |dz|,\quad z\in\gamma. $$    
    Hence, on $\gamma$,
    \begin{align*}
        -2\pi w_{\gamma,V} = \partial_{\mathfrak{n}_+}U_{\mu_{\gamma,V}} + \partial_{\mathfrak{n}_-}U_{\mu_{\gamma,V}}  = -\partial_{\mathfrak{n}_+}\log |\psi| - \frac{1}{2}( \partial_{\mathfrak{n}_+}V_{\text{ext}} + \partial_{\mathfrak{n}_-}V_{\text{int}}).
    \end{align*}
    As a consequence, $w_{\gamma,V}\in C^{m-1,\alpha}(\gamma)$ and thus $F_{\gamma,V} \in C^{m,\alpha}([a_\gamma,b_\gamma])$. Similarly, we can show that $F_{\tau}\in C^{m,\alpha}([a_\tau,b_\tau])$ (replace $\gamma\mapsto\tau$ and $V\mapsto 0$). Since $F_{\gamma,V}'(t) = (w_{\gamma,V}\gamma)(t)>0$ for all $t\in [a_\gamma,b_\gamma]$, we also have $F_{\gamma,V}^{-1} \in C^{m,\alpha}([0,1])$. Hence, by definition \eqref{eq_par_def}, also $z_e\in C^{m,\alpha}(\tau)$. Furthermore, it follows from \eqref{z_e_der} and continuity of $F_{\tau}'$ and $F_{\gamma,V}'$ that
    $$|(z_e\tau)'(t)|\, \geq \frac{\min_{\tau}w_\tau}{\max_{\gamma}w_{\gamma,V}} > 0,$$
    which proves the final assertion.
\end{proof}

\subsection{Grunsky operator}

The standing assumption in this subsection will be the following. The stated $L^2$-space, and all other $L^2$-spaces in this manuscript, only contains real-valued functions.

\begin{assum} \label{ass1_L2}
Assumption \ref{ass2} holds and the $\log$-kernel
    $$\mathcal{L}_0:\tau\times\tau \to\R\cup\{\infty\}:(u,v)\mapsto \log\frac{1}{|u-v|},  $$
and the Grunsky kernel
$$\mathcal{G}:\tau\times\tau \to\R\cup\{\infty\}:(u,v)\mapsto \log\left|\frac{u-v}{z_e(u)-z_e(v)}\right| + \frac{1}{2} ((Vz_e)(u)+(Vz_e)(v)),$$
are both in $L^2(\tau\times\tau,\mu_\tau\otimes\mu_\tau)$.
\end{assum}

Assumption \ref{ass1_L2} will allow us to expand $\mathcal{L}_0$ and $\mathcal{G}$ in an appropriate $L^2$-basis, which in its turn will allow us to develop some general theory for the induced operators on $l^2$. As indicated by, e.g., Proposition \ref{ep_reg}, Assumption \ref{ass1_L2} reflects some initial regularity assumed on the objects $\gamma$ and $\tau$.
\medbreak

It will be convenient to introduce the closed subspace 
$$L_0^2(\tau,\mu_\tau) = \{f\in L^2(\tau,\mu_\tau) : \int f d\mu_\tau = 0\},$$ 
of the Hilbert space $L^2(\tau,\mu_\tau)$. Since $\mathcal{L}_0\in L^2(\tau\times\tau,\mu_\tau\otimes\mu_\tau)$ is symmetric and real-valued, the single-layer potential operator
    $$ (U_{\mu_\tau}f)(v) = \int f(u) \log\frac{1}{|u-v|} d\mu_\tau(u),  $$
is a Hilbert-Schmidt integral operator on $L^2(\tau,\mu_\tau)$. Its restriction to $L_0^2(\tau,\mu_\tau)$ is therefore Hilbert-Schmidt as well. By the Hilbert-Schmidt theorem, there then exists an orthonormal basis $\phi = (\phi_n)_{n\geq 1}$ of $L_0^2(\tau,\mu_\tau)$, with $U_{\mu_\tau}\phi_n=\lambda_n \phi_n$ and all $\lambda_n\in\R$, such that for all $f\in L_0^2(\tau,\mu_\tau)$,
\begin{equation} \label{Uf_exp}
     U_{\mu_\tau}f = \sum_{n\geq 1} \lambda_n \left(\int f \phi_n d\mu_{\tau} \right) \phi_n.
\end{equation}
In particular,
    $$ \int (U_{\mu_\tau}\phi_k)(v) \phi_l(v) d\mu_{\tau}(v) = \lambda_k \delta_{k,l},\quad k,l\geq 1. $$
We may extend $\phi$ to an orthonormal basis $(\phi_n)_{n\geq 0}$ of $L^2(\tau,\mu_\tau)$ by setting $\phi_0=1$. In that case, $\mathcal{L}_0\in L^2(\tau\times\tau,\mu_\tau\otimes\mu_\tau)$ admits an expansion of the form
    $$ \mathcal{L}_0(u,v) = \Lambda_{0,0} + \Lambda_1 \phi(u)^t + \Lambda_2 \phi(v)^t + \phi(u)\Lambda \phi(v)^t,\quad \text{a.e. } u,v\in\tau, $$
where due to \eqref{Uf_exp}, we have $\Lambda=[\lambda_k\delta_{k,l}]_{k,l\geq 1}$. It follows from Parseval's theorem that
    $$ \sum_{k\geq 1} \lambda_k^2 < \infty, $$
so that, in particular, $\lim_{k\to\infty} |\lambda_k| = 0$.
\medbreak

We will now show that $\Lambda_1=\Lambda_2=0$ in the expansion.

\begin{prop} \label{L_0_exp}
    If $\mathcal{L}_0\in L^2(\tau\times\tau,\mu_\tau\otimes\mu_\tau)$, then
    $$ \mathcal{L}_0(u,v) = I[\mu_\tau] + \phi(u) \Lambda \phi(v)^t,\quad \text{a.e. } u,v\in\tau. $$
\end{prop}
\begin{proof}
    Since $\phi$ is an orthonormal basis, we have 
        $$ \int \log\frac{1}{|u-v|} d\mu_\tau(u) =   \Lambda_{0,0} + \Lambda_2 \phi(v)^t,\quad v\in\text{supp}(\mu_\tau).$$
    On the other hand, we know from potential theory that it needs to be constant along $v\in\text{supp}(\mu_\tau)$. Since $\text{supp}(\mu_\tau)=\tau$, we can therefore conclude that $\Lambda_2=0$. By symmetry, $\Lambda_1=0$ as well. Integrating along $d\mu_\tau(v)$ in the above then yields 
        $$ \Lambda_{0,0} = \iint \log\frac{1}{|u-v|} d\mu_\tau(u)d\mu_\tau(v), $$
    which is the desired expression for $\Lambda_{0,0}$.
\end{proof}

We may interpret $\Lambda$ as a bounded operator $\ell^2(\Z_{\geq 1})\to\ell^2(\Z_{\geq 1})$ by acting from the right on $\vec{f}\in \ell^2(\Z_{\geq 1})$ through multiplication. The following result then shows that it is positive.

\begin{prop} \label{log_SPD}
If $\mathcal{L}_0\in L^2(\tau\times\tau,\mu_\tau\otimes\mu_\tau)$, then $\vec{f}\Lambda\vec{f}^t>0$ for all $\vec{f}\in \ell^2(\Z_{\geq 1})$. In particular, all $\lambda_n>0$.
\end{prop}
\begin{proof}
    By variations of the energy functional $\mu\mapsto I_{\tau,0}[\mu]$, we know that
        $$ \iint f(u) f(v) \log \frac{1}{|u-v|} d\mu_\tau(u) d\mu_\tau(v) \geq 0, $$
    for every $f\in L_0^2(\tau,\mu_\tau)$. In the basis $\phi$, this becomes
        $$ \iint \vec{f} \phi(u)^t (\Lambda_{0,0} + \phi(u) \Lambda \phi(v)^t) \phi(v) \vec{f}^t d\mu_\tau(u) d\mu_\tau(v) \geq 0. $$
    The orthonormality conditions $\int \phi(u)^t \phi(u) d\mu_\tau(u) = I $ then imply that $ \vec{f} \Lambda \vec{f}^t \geq 0 $. Suppose now that $\vec{f} \Lambda \vec{f}^t = 0$, i.e.
        $$ \iint f(u) f(v) \log \frac{1}{|u-v|} d\mu_\tau(u) d\mu_\tau(v) = 0. $$
    It then follows from non-degeneracy of the energy functional that we must have $fd\mu_\tau=0$ and thus $f=0$ on $\text{supp}(\mu_\tau)$. Since $\text{supp}(\mu_\tau)=\tau$, we can then conclude that $\vec{f}=0$.
\end{proof}

By the previous result, we can also expand $\mathcal{L}_0$ in the rescaled basis $\psi = (\sqrt{\lambda_n} \phi_n)_{n\geq 1}$ as
    $$ \mathcal{L}_0(u,v) = L_{0,0} + \psi(u) \psi(v)^t,\quad \text{a.e. } u,v\in\tau. $$
Such an expansion will play a important role later.
\medbreak

Through the equilibrium parametrization, we can study the $\log$-kernel on $\gamma$ in the external field $V$ by parameterizing it on $\tau$. We define 
    $$\mathcal{L}_1:\tau\times\tau \to\R\cup\{\infty\}:(u,v)\mapsto \log\frac{1}{|z_e(u)-z_e(v)|} + \frac{1}{2} ((Vz_e)(u)+(Vz_e)(v)).$$
Since $\mathcal{L}_1=\mathcal{L}_0+\mathcal{G}$, clearly $\mathcal{L}_1\in L^2(\tau\times\tau,\mu_\tau\otimes\mu_\tau)$ by Assumption \ref{ass1_L2}. As a consequence, we can expand $\mathcal{L}_1$ in the same basis as $\mathcal{L}_0$. Say we have the expansion
    $$ \mathcal{L}_1(u,v) = L_{0,0} + L_1 \phi(u)^t + L_2 \phi(v)^t + \phi(u) L \phi(v)^t,\quad \text{a.e. } u,v\in\tau. $$
In the following lemma, we will show that $L_1=L_2=0$. 

\begin{prop} \label{L_1_exp}
    If $\mathcal{L}_1\in L^2(\tau\times\tau,\mu_\tau\otimes\mu_\tau)$, then
    $$ \mathcal{L}_1(u,v) = I_{\gamma,V}[\mu_{\gamma,V}] + \phi(u) L \phi(v)^t,\quad \text{a.e. } u,v\in\tau. $$
\end{prop}
\begin{proof}
    Since $\phi$ is an orthonormal basis, we have
    \begin{align*}
        L_2 = \iint &\left(\log\frac{1}{|z_e(u)-z_e(v)|} + \frac{1}{2} ((Vz_e)(u)+(Vz_e)(v)) \right) \phi(v) d\mu_\tau(u)d\mu_\tau(v).
    \end{align*}
    On the other hand, we know from potential theory that
    \begin{align*}
        \int &\left(\log\frac{1}{|z_e(u)-z_e(v)|} +  \frac{1}{2} (Vz_e)(v) \right) d\mu_\tau(u) \\ 
        &= \int \left(\log\frac{1}{|u-z_e(v)|} +  \frac{1}{2} V(v) \right) d\mu_{\gamma,V}(u) \\ &= I_V[\mu_{\gamma,V}] - \frac{1}{2} \int V d\mu_{\gamma,V},
    \end{align*}
    for all $v\in\tau$ with $z_e(v)\in\text{supp}(\mu_{\gamma,V})$. Since $\text{supp}(\mu_{\gamma,V})=\gamma$, the latter is certainly satisfied for all $v\in\tau$. We can therefore conclude that
        $$L_2 = \int \left(I_{\gamma,V}[\mu_{\gamma,V}] - \frac{1}{2} \int V d\mu_{\gamma,V} + \frac{1}{2} \int (Vz_e) d\mu_\tau\right) \phi(v) d\mu_\tau(v) = 0, $$
    By symmetry, $L_1=0$ as well.  Finally,
        $$ L_{0,0} = \iint \mathcal{L}_1(u,v) d\mu_\tau(u)d\mu_\tau(v), $$
    and thus $L_{0,0} = I_V[\mu_{\gamma,V}]$.
\end{proof}

The Grunsky kernel $\mathcal{G}$ arises by comparing the $\log$-kernel on $\tau$ with the $\log$-kernel on $\gamma$ parametrized by the equilibrium parametrization. More precisely, we define
    $$ \mathcal{G}(u,v) = \log\left|\frac{u-v}{z_e(u)-z_e(v)}\right| + \frac{1}{2} ((Vz_e)(u)+(Vz_e)(v)),\quad \text{a.e. } u,v\in\tau. $$
We thus have,
    $$\mathcal{G}(u,v) = \mathcal{L}_1(u,v) - \mathcal{L}_0(u,v),\quad \text{a.e. } u,v\in\tau, $$
and
    $$\mathcal{G}(u,v) =  G_{0,0} + \phi(u) G \phi(v)^t,\quad \text{a.e. } u,v\in\tau,$$ 
with 
    $$ G_{0,0} = I_V(\mu_{\gamma,V})-I(\mu_\tau),\quad G=L-\Lambda. $$

In the rescaled basis $\psi = \Lambda^{\frac{1}{2}} \phi^t = (\sqrt{\lambda_n}\phi_n)_{n\geq 1}$, the kernel $\mathcal{G}(u,v)$ can be expressed as 
\begin{equation} \label{Guvexp}
    \mathcal{G}(u,v) =  G_{0,0} + \psi(u) K^{\gamma,V}_{\tau} \psi(v)^t,\quad \text{a.e. } u,v\in\tau,
\end{equation}
in terms of 
$$K^{\gamma,V}_{\tau} = \Lambda^{-\frac{1}{2}} G \Lambda^{-\frac{1}{2}} = \big[G_{k,l}/\sqrt{\lambda_k\lambda_l}\big]_{k,l\geq 1}.$$
The representation \eqref{Guvexp} should be thought of as the weighted Grunsky expansion and the coefficients $(K^{\gamma,V}_{\tau})_{k,l}$ as the weighted Grunsky coefficients.

\begin{rem}
    For a Jordan curve $\gamma$, $V=0$ and $\tau=\mathbb{T}$, this is exactly the real version of the classical Grunsky expansion described in \eqref{ReGr}. We will carefully work this out in the next section. For a Jordan arc $\gamma$, $V=0$ and $\tau=[-1,1]$, this is the arc-Grunsky expansion in \cite[Lem. 1.4]{CourteautJohanssonViklund2026}.
\end{rem}

Both $\Lambda$ and $G$ can be seen as operators $\ell^2(\Z_{\geq 1})\to\ell^2(\Z_{\geq 1})$ by acting from the right on sequences $\vec{f}\in\ell^2(\Z_{\geq 1})$ through multiplication. In fact, due to Assumption \ref{ass1_L2}, they are both self-adjoint Hilbert-Schmidt operators. The same does not need to hold for $K^{\gamma,V}_{\tau}$, which can generally only be interpreted as an operator $\ell^2(\Z_{\geq 1})\Lambda^{\frac{1}{2}}\to\ell^2(\Z_{\geq 1})\Lambda^{-\frac{1}{2}}$. Here $\ell^2(\Z_{\geq 1})\Lambda^{\frac{1}{2}}\subset\ell^2(\Z_{\geq 1})$ as $\lim_{k\to\infty} \lambda_k=0$. In order to extend it to the whole of $\ell^2(\Z_{\geq 1})$, we need to add some assumptions on the decay of the Grunsky coefficients. In the next section we will show that this can certainly be done if $\gamma$ is a sufficiently regular Jordan curve, $V$ is a sufficiently regular potential and $\tau=\mathbb{T}$ via Proposition \ref{ep_reg}. For now, we obtain the following generalized Grunsky inequality on the smaller subspace $\ell^2(\Z_{\geq 1})\Lambda^{\frac{1}{2}}$.

\begin{prop}[Grunsky inequality] \label{GI}
    Suppose that Assumption \ref{ass1_L2} holds. Then, for all non-zero $\vec{f}\in \ell^2(\Z_{\geq 1})$, we have $\vec{f} G\vec{f}^t>-\vec{f} \Lambda\vec{f}^t$. In particular, $K^{\gamma,V}_{\tau}>-I$ on $\ell^2(\Z_{\geq 1})\Lambda^{\frac{1}{2}}$.
\end{prop}
\begin{proof}
    We have to show that $L>0$. Let $f\in L_0^2(\tau,\mu_\tau)$ and note that
    \begin{equation} \label{EP_COV_2}
        \int fz_e^{-1} d\mu_{\gamma,V} = \int f d\mu_\tau =0,
    \end{equation}
    as $\mu_{\gamma,V} = z_e\#\mu_\tau$. Hence, by variations of the energy functional $\mu \mapsto I_{\gamma,0}[\mu]$, we have
        $$ \iint \log \frac{1}{|u-v|} (fz_e^{-1})(u) (fz_e^{-1})(v) d\mu_{\gamma,V}(u) d\mu_{\gamma,V}(v) \geq 0. $$
    On the other hand, by \eqref{EP_COV_2}, 
     $$ \iint (V(u)+V(v)) (fz_e^{-1})(u) (fz_e^{-1})(v) d\mu_{\gamma,V}(u) d\mu_{\gamma,V}(v) = 0. $$
    It then follows from $\mu_{\gamma,V} = z_e\#\mu_\tau$ that
    $$ \iint \mathcal{L}_1(u,v) f(u) f(v) d\mu_\tau(u) d\mu_\tau(v) \geq 0, $$
    In the Fourier basis, with $f = \vec{f}^t \phi$, this becomes $ \vec{f}^t L \vec{f} \geq 0$. Suppose that
        $$ \iint \mathcal{L}_1(u,v) f(u) f(v) d\mu_\tau(u) d\mu_\tau(v) =0, $$
    then, as $\mu_{\gamma,V} = z_e\#\mu_\tau$, 
     $$ \iint \left(\log \frac{1}{|u-v|} + \frac{1}{2} (V(u)+V(v)) \right) (fz_e^{-1})(u) (fz_e^{-1})(v) d\mu_{\gamma,V}(u) d\mu_{\gamma,V}(v) =0, $$
     and hence also
     $$ \iint \log \frac{1}{|u-v|} (fz_e^{-1})(u) (fz_e^{-1})(v) d\mu_{\gamma,V}(u) d\mu_{\gamma,V}(v) =0. $$
     Non-degeneracy of the energy functional $\mu \mapsto I_{\gamma,0}[\mu]$, then implies that $(fz_e^{-1}) d\mu_{\gamma,V}=0$ and thus $(fz_e^{-1})=0$ on $\text{supp}(\mu_{\gamma,V})$. Since $\text{supp}(\mu_{\gamma,V}) = \gamma$ and $z_e:\tau\to\gamma$ is bijective, we can then conclude that $f=0$ on $\tau$ and thus $\vec{f}=0$.  
\end{proof}

Under some additional assumptions on $G$, we can further strengthen this inequality.

\begin{prop}[Strengthened Grunsky inequality] \label{SGI}
    Suppose that Assumption \ref{ass1_L2} holds and that $G\Lambda^{-1}:\ell^2(\Z_{\geq 1})\to\ell^2(\Z_{\geq 1})$ is a Hilbert-Schmidt operator. Then $K^{\gamma,V}_{\tau}$ defines a self-adjoint Hilbert-Schmidt operator on $\ell^2(\Z_{\geq 1})$ and there exists $\kappa\in(0,1)$ such that $K^{\gamma,V}_{\tau} \geq -\kappa I$ on $\ell^2(\Z_{\geq 1})$.
\end{prop}
\begin{proof}
    For notional convenience, denote $K=K^{\gamma,V}_{\tau}$. Since $G_{k,l}=G_{l,k}$,
    \begin{align*}
        \sum_{k,l\geq1}|K_{k,l}|^2
        =\sum_{k,l\geq1}\frac{|G_{k,l}|^2}{\lambda_k\lambda_l}
        \leq\frac12\sum_{k,l\geq1}|G_{k,l}|^2
          \left(\lambda_k^{-2}+\lambda_l^{-2}\right)
        =\|G\Lambda^{-1}\|_{\rm HS}^2<\infty.
    \end{align*}
    Hence $K$ is a self-adjoint Hilbert-Schmidt operator on $\ell^2(\Z_{\geq1})$. Let $x\in\ell^2(\Z_{\geq1})$ be finitely supported and set $f=x\Lambda^{-\frac{1}{2}}$. It follows from Proposition \ref{GI} that
    $$ \langle x,x(I+K)\rangle_{\ell^2}
        =\langle f,f(\Lambda+G)\rangle_{\ell^2}>0.$$
    By denseness and boundedness, $I+K\geq0$ on $\ell^2(\Z_{\geq1})$. Furthermore, if $x(I+G\Lambda^{-1})=0$, then $x(\Lambda+G)=0$ and thus $x=0$ by Proposition \ref{GI}. Since $G\Lambda^{-1}$ is compact, the Fredholm alternative theorem then shows that $I+G\Lambda^{-1}:\ell^2(\Z_{\geq1})\to\ell^2(\Z_{\geq1})$ is an invertible operator with bounded inverse. Hence, the same holds for $(I+G\Lambda^{-1})^{\ast}=I+\Lambda^{-1}G$. It then follows from the identity
    $$(I+K)\Lambda^{\frac{1}{2}} = \Lambda^{\frac{1}{2}}(I+\Lambda^{-1}G),$$
    that $I+K$ is injective. Indeed, if $x(I+K)=0$ then $x\Lambda^{\frac{1}{2}}(I+\Lambda^{-1}G)=0$ so that $x\Lambda^{\frac{1}{2}}=0$ and thus $x=0$. Since $K$ is compact, every nonzero element of the spectrum $\sigma(K)$ of $K$ is an eigenvalue. Hence, we must have that $-1\notin\sigma(K)$. Together with $I+K\geq0$, this implies that $I+K\geq\delta I$ for some $\delta>0$. The latter then certainly holds for some $\delta\in(0,1)$ as well. Taking $\kappa=1-\delta\in(0,1)$ then proves the desired result.
\end{proof}

\begin{rem}
    For a Jordan curve $\gamma$, $V=0$ and $\tau=\mathbb{T}$, this gives one side of the classical strengthened Grunsky inequality, see \cite[\S 9.4]{Pommerenke1975}. In this setting, the Hilbert-Schmidt assumption is a sufficient condition: the optimal condition to have the classical two-sided ineqality is that $\gamma$ is a quasicircle, see \cite{Kuhnau1982}. For a Jordan arc $\gamma$, $V=0$ and $\tau=[-1,1]$, the above is the strengthened arc-Grunsky inequality in \cite[Lem. 1.4]{CourteautJohanssonViklund2026}.
\end{rem}

In what follows, the interpolating log-kernel
\begin{equation*} \label{L_s}
    \mathcal{L}_s(u,v) = (1-s) \mathcal{L}_0(u,v) + s \mathcal{L}_1(u,v),\quad s\in[0,1],
\end{equation*}
which interpolates between $\mathcal{L}_0$ and $\mathcal{L}_1$ will play an important role. Observe that
\begin{equation} \label{L_s:G}
    \mathcal{L}_s(u,v) = \mathcal{L}_0(u,v) + s \mathcal{G}(u,v),
\end{equation}
hence, w.r.t. the basis $\psi$, we have
\begin{equation} \label{L_s_exp}
    \mathcal{L}_s(u,v) = I_{\tau}[\mu_{\tau}] + s G_{0,0} + \psi(u) L_s \psi(v)^t,\quad \text{a.e. } u,v\in\tau.
\end{equation}
where
    $$ L_s = I + s K^{\gamma,V}_{\tau}. $$
    
\begin{prop} \label{L_s_inv}
    Suppose that Assumption \ref{ass1_L2} holds and that $G\Lambda^{-1}:\ell^2(\Z_{\geq 1})\to\ell^2(\Z_{\geq 1})$ is a Hilbert-Schmidt operator. Then $L_s=I+sK^{\gamma,V}_{\tau}$ defines an invertible operator $\ell^2(\Z_{\geq 1})\to \ell^2(\Z_{\geq 1})$ with bounded inverse for all $s\in[0,1]$. Moreover,
    $$ \sup_{s\in[0,1]}\|L_s^{-1}\|\leq\frac{1}{1-\kappa}.$$
\end{prop}
\begin{proof}
     This follows immediately from Proposition \ref{SGI} as $L_s\geq(1-s\kappa)I\geq(1-\kappa)I$.
\end{proof}

In what follows and throughout the next sections, we will use the convention that, for every function $f\in C^1(\tau)$,
\begin{equation} \label{tan_der}
    f^{(1)}(\tau(t)) = \frac{(f\tau)'(t)}{|\tau'(t)|}=(f\tau)'(t),\quad t\in [a_\tau,b_\tau].
\end{equation}
We will use the same convention for higher order derivatives and functions in several variables.
\medbreak

It turns out that the operator $L_s$ can be used to solve the following functional equation, which we refer to as the master equation, 
\begin{equation} \label{ME}
    \frac{1}{\beta}g(u)= {\rm p.v.} \int h_s(v) \mathcal{L}_s^{(0,1)}(u,v) d\mu_{\tau}(v),\quad \text{a.e. } u\in\tau.
\end{equation} 
Later, the master equation will play an essential role in our analysis. In order to show that it admits a solution given a function $g$ (and for the equation to make sense), we need to be able to take a derivative in the $L^2$-space, i.e. we will assume that all $\phi_n^{(1)}\in L^2(\tau,\mu_\tau)$ so that we have an expansion
\begin{equation} \label{tau_D}
    \phi^{(1)}(v)^t = \vec{d}_0^t + D \phi(v)^t,\quad \text{a.e. } u\in\tau,
\end{equation}
where 
\begin{align*}
    \vec{d}_0 = \int \phi^{(1)} d\mu_\tau, \quad D = \int (\phi^{(1)})^t \phi d\mu_\tau.
\end{align*}

\begin{prop} \label{ME_FOUR_SOL}
    Suppose that Assumption \ref{ass1_L2} holds, $\Lambda^{-1}G:\ell^2(\Z_{\geq 1})\to\ell^2(\Z_{\geq 1})$ is a Hilbert-Schmidt operator and $\Lambda^{\frac{1}{2}}D^t\Lambda^{\frac{1}{2}}:\ell^2(\Z_{\geq 1})\to\ell^2(\Z_{\geq 1})$ is an invertible operator with bounded inverse. If $g=\vec{g}\phi^t$ with $\vec{g}\Lambda^{-\frac{1}{2}}\in\ell^2(\Z_{\geq 1})$, there exists $h_s=\vec{h}_s\phi^t$ with $\vec{h}\Lambda^{-\frac{1}{2}}\in\ell^2(\Z_{\geq 1})$ such that \eqref{ME} is satisfied. Moreover,
        $$ \vec{h}_s \Lambda^{-\frac{1}{2}} = \frac{1}{\beta} (\vec{g} \Lambda^{-\frac{1}{2}}) L_s^{-1} (\Lambda^{\frac{1}{2}}D^t\Lambda^{\frac{1}{2}})^{-1} . $$
\end{prop}
\begin{proof}
It follows from \eqref{L_s_exp} that the right hand side of \eqref{ME} is given by 
    $$ \int \phi(u) (\Lambda+sG) \phi^{(1)}(v)^t \phi(v) \vec{h}_s^t d\mu_\tau(v).  $$
The master equation thus becomes
    $$ \frac{1}{\beta} \phi(u) \vec{g}^t = \phi(u) (\Lambda+sG) D \vec{h}_s^t,\quad \text{a.e. } u\in\tau,  $$
or equivalently,
\begin{equation} \label{ME_FOUR}
    \frac{1}{\beta} \vec{g} =  \vec{h}_s D^t (\Lambda^{\frac{1}{2}}L_s\Lambda^{\frac{1}{2}}).
\end{equation}
Since both $L_s$, see Proposition \ref{L_s_inv}, and $\Lambda^{\frac{1}{2}}D^t\Lambda^{\frac{1}{2}}$ are invertible, the solution of the master equation is therefore given by
    $$ \vec{h}_s \Lambda^{-\frac{1}{2}} = \frac{1}{\beta} (\vec{g} \Lambda^{-\frac{1}{2}}) L_s^{-1} (\Lambda^{\frac{1}{2}}D^t\Lambda^{\frac{1}{2}})^{-1} . $$
As $\vec{g}\Lambda^{-\frac{1}{2}}\in\ell^2(\Z_{\geq 1})$, this also implies that $\vec{h}\Lambda^{-\frac{1}{2}}\in\ell^2(\Z_{\geq 1})$.
\end{proof}

\subsection{Example}

Recall that we can expand any function
$f \in L^2(\mathbb{T}; \mu_{\mathbb{T}})$
in its Fourier series
\[f(z)=\sum_{k \in \mathbb{Z}} \hat{f}_k z^k,\quad \text{a.e. } z \in \mathbb{T},\]
as the set $\{ z^k : k \in \mathbb{Z} \}$ is an orthonormal basis for $L^2(\mathbb{T};\mu_{\mathbb{T}})$. In the variable $z = e^{i\theta}$, this induces an expansion in sines and cosines
\begin{equation*}
    f(e^{i\theta}) = \hat{f}_0 + \sum_{k\geq 1} [(\hat{f}_k+\hat{f}_{-k}) \cos(k\theta) + i(\hat{f}_k-\hat{f}_{-k})\sin(k\theta)],\quad \text{a.e. } \theta\in[0,2\pi].
\end{equation*}
The advantage of such an expansion is that it involves a real-valued basis. Denote 
$$ \phi_k(e^{i\theta}) = \begin{dcases}
      \sqrt{2}\cos(|k|\theta),\quad k\in\Z_{>0}, \\
      \sqrt{2} \sin(|k|\theta),\quad k\in\Z_{<0}.
\end{dcases} $$
Then $\phi = (\phi_k)_{k\in\Z_{\neq 0}}$ is an orthonormal basis for $L_0^2(\mathbb{T},\mu_{\mathbb{T}})$.
\medbreak

The $\log$-kernel on $\mathbb{T}$ takes the special form
    $$ \mathcal{L}_0(e^{iu},e^{iv}) = - \log\left| 2 \sin\left(\frac{u-v}{2}\right) \right|,\quad u,v\in[0,2\pi]: u\neq v. $$
We may expand $\mathcal{L}_0$ in terms of $\phi$ as follows, which shows that $\mathcal{L}_0\in L^2(\mathbb{T}^2,\mu_{\mathbb{T}}\otimes\mu_{\mathbb{T}})$.

\begin{prop}
    Define $\Lambda= [\frac{1}{2|k|} \delta_{k,l}]_{k,l\in\Z_{\neq 0}}$, then
    \begin{equation*}
        \mathcal{L}_0(e^{iu},e^{iv}) = \phi(u) \Lambda \phi(v)^t,\quad u,v\in[0,2\pi]: u\neq v.
    \end{equation*}
    In particular,
    \begin{equation} \label{T_log_EV}
        \int \phi_n(u) \mathcal{L}_0(u,v) d\mu_{\mathbb{T}}(u) = \frac{1}{2|k|} \phi_n(v).
    \end{equation}
\end{prop}
\begin{proof}
    The above follows from the fact that
        $$\mathcal{L}_0(e^{iu},e^{iv}) = -\log\left|1-e^{i(u-v)}\right| = \text{Re\,}\sum_{k\geq 1} \frac{e^{ik(u-v)}}{k},  $$ 
    and the identity $\cos(k(u-v))=\cos(ku)\cos(kv)+\sin(ku)\sin(kv)$.
\end{proof} 

It follows from \eqref{T_log_EV} that $\psi = (\psi_k)_{k\in\Z_{\neq 0}}$ with 
$$ \psi_k(e^{i\theta}) = \begin{dcases}
    |k|^{-\frac{1}{2}}\cos(|k|\theta),\quad k\in\Z_{>0}, \\
    |k|^{-\frac{1}{2}}\sin(|k|\theta),\quad k\in\Z_{<0},
\end{dcases} $$
is an orthogonal basis for $L_0^2(\mathbb{T},\mu_{\mathbb{T}})$ with the property that
$$ \log\frac{1}{|u-v|} = \psi(u) \psi(v)^t. $$

Whenever $\gamma$ is a $C^{1,\alpha}$ Jordan curve and $V\in C^{1,\alpha}(\gamma)$, we have $z_e\in C^{1,\alpha}(\mathbb{T})$ and $|z_e^{(1)}|\,\geq c>0$ by Proposition \ref{ep_reg}, and hence $\mathcal{G}\in C^0(\mathbb{T}^2)$. In that case, $\mathcal{G}\in L^2(\mathbb{T}^2,\mu_{\mathbb{T}}\otimes \mu_{\mathbb{T}})$ and it admits an expansion of the form
    $$ \mathcal{G}(u,v) =  G_{0,0} + \phi(u) G \phi(v)^t,\quad u,v\in\tau, $$
or alternatively,
    $$ \mathcal{G}(u,v) =  G_{0,0} + \psi(u) K^{\gamma,V}_{\mathbb{T}} \psi(v)^t,\quad u,v\in\tau, $$
in terms of $K^{\gamma,V}_{\mathbb{T}} = [2|kl|^{\frac{1}{2}} G_{k,l}]_{k,l\in\Z_{\neq 0}}$. At this level of regularity, Proposition \ref{GI} only gives $K^{\gamma,V}_{\mathbb{T}}>-I$ on $\ell^2(\Z_{\geq 1})\Lambda^{\frac{1}{2}}$. To apply Proposition \ref{SGI}, we must verify that $G\Lambda^{-1}$ is a Hilbert-Schmidt operator on $\ell^2(\Z_{\geq 1}$. Additional regularity of $z_e$ will give the required decay of the Grunsky coefficients to prove this.

\begin{prop} \label{reg_G}
    Let $m\in\Z_{\geq 1}$ and suppose that $z_e\in C^{m,\alpha}(\mathbb{T})$ and $V\in C^{m-1,\alpha}(\gamma)$. If $p,q\geq0$ and $p+q=m+\alpha-1$, then there exists $C_{p,q}>0$ such that for all $k,l\in\Z_{\neq 0}$,
        $$ |G_{k,l}| \leq \frac{C_{p,q}}{|k|^{p}|l|^{q}}. $$
\end{prop}
\begin{proof}
    Since
        $$ \frac{z_e(e^{iu})-z_e(e^{iv})}{e^{iu}-e^{iv}} = \frac{v-u}{e^{iu}-e^{iv}} \int_0^1 z_e^{(1)}(e^{i(v+a(v-u))}) da,\quad u,v\in\mathbb{T}. $$
    the assumptions imply that $\mathcal{G}\in C^{m-1,\alpha}(\mathbb{T}^2)$. Hence, after integrating $m-1$ times in the coordinate with the larger Fourier frequency and using the standard translation estimate for a H\"older continuous function, we obtain
    $$|G_{k,l}|\leq C_{m,\alpha}\max\{|k|,|l|\}^{-(m-1+\alpha)}.$$
    Since $|k|^p|l|^q\leq\max\{|k|,|l|\}^{p+q}$, the desired result then follows.
\end{proof}

The strengthened Grunsky inequality in Proposition \ref{SGI} can then be proven under the following assumptions.

\begin{prop} \label{SGI_UC}
    Suppose that $z_e\in C^{3,\alpha}(\mathbb{T})$ and $V\in C^{2,\alpha}(\gamma)$. Then $K^{\gamma,V}_{\mathbb{T}}$ defines a self-adjoint Hilbert-Schmidt operator on $\ell^2(\Z_{\neq 0})$ and there exists $\kappa\in(0,1)$ such that $K^{\gamma,V}_{\mathbb{T}} \geq -\kappa I$ on $\ell^2(\Z_{\neq 0})$.
\end{prop}
\begin{proof}
    Here $\lambda_k=(2|k|)^{-1}$. Applying Proposition \ref{reg_G} with
        $p=\frac{3}{2}+\frac{\alpha}{2}$ and $q=\frac{1}{2}+\frac{\alpha}{2}$,
    gives
        $$\|G\Lambda^{-1}\|_{\rm HS}^2
        =4\sum_{k,l\in\Z_{\neq0}}|k|^2|G_{k,l}|^2
        \leq C\sum_{k,l\in\Z_{\neq0}}|k|^{-1-\alpha}|l|^{-1-\alpha}<\infty.$$
    The desired result then follows from Proposition \ref{SGI}.
\end{proof}

For the Fredholm determinant $\det(I+K^{\gamma,V}_{\mathbb{T}})$ to be well-defined, the operator $K^{\gamma,V}_{\mathbb{T}}$ also needs to be trace class. If $z_e\in C^{4,\alpha}(\mathbb{T})$ and $V\circ z_e\in C^{3,\alpha}(\mathbb{T})$, Proposition \ref{reg_G} with $p=q=(3+\alpha)/2$ gives
    $$\sum_{k,l\neq0}|K_{k,l}|
    \leq C\sum_{k,l\neq0}|k|^{-1-\alpha/2}|l|^{-1-\alpha/2}<\infty,$$
so that $K$ is trace class. Proposition \ref{SGI_UC} then implies that $\det(I+K)>0$. In particular, the quantity $\log\det(I+K_{\mathbb{T}}^{\gamma,V})$ in Theorem \ref{PF} is well-defined under its stated hypotheses.
\medbreak

We will now have a closer look at the master equation \eqref{ME} on $\tau=\mathbb{T}$. On $\mathbb{T}$, we have $\vec{d}_0 = 0$ and
\begin{equation} \label{T_D} 
    D = \int (\phi^{(1)})^t \phi d\mu_\tau  = \begin{pmatrix} 0 & -\left[ k \delta_{k,l}\right]_{k,l\geq 1} \\ -\left[ k \delta_{k,l}\right]_{k,l\leq -1} & 0 \end{pmatrix}.
\end{equation}
Note that $D^t = -D$, which leads to the integration by parts rule
\begin{equation} \label{IBP}
    \int f g^{(1)} d\mu_{\tau} = - \int f^{(1)} g d\mu_{\tau},
\end{equation}
for $f,g \in L_0^2(\tau,\mu_\tau)$. Indeed, write $f(u) = \vec{f} \phi(u)^t$ and  $g(u) = \vec{g} \phi(u)^t$, then
\begin{align*}
f^{(1)}(u) =\vec{f} D \phi(u)^t, \quad g^{(1)}(u) = \vec{g} D \phi(u)^t.
\end{align*}
Hence,
\begin{align*}
\int f g^{(1)} d\mu_\tau = f_0 \vec{d}_0 \vec{g}^t + \vec{f} D^t \vec{g}^t,\quad \int f^{(1)} g d\mu_\tau &= \vec{f} \vec{d}_0^t g_0 + \vec{f} D \vec{g}^t,
\end{align*}
and thus
$$\int f g^{(1)} d\mu_\tau = - \int f^{(1)} g d\mu_\tau.$$
This integration by parts rule will play an important role in Section \ref{S_RV}. 
\medbreak

The result below shows how regularity of $g:\mathbb{T}\to\R$ corresponds to regularity of the associated solution $h_s:\mathbb{T}\to\R$ of the master equation \eqref{ME}. Note that
    $$ \tilde{D} = \Lambda^{\frac{1}{2}} D \Lambda^{\frac{1}{2}} = \frac{1}{2} \begin{pmatrix} 0 & -I \\ I & 0 \end{pmatrix}. $$
is invertible with inverse $\tilde{D}^{-1} = - 4 \tilde{D} $.

\begin{prop} \label{reg_h_s}
    Let $m\in\Z_{\geq 1}$ and suppose that $z_e\in C^{m+3,\alpha}(\mathbb{T})$ and $V\in C^{m+2,\alpha}(\gamma)$. If $g\in C^{m,\alpha}(\gamma)$ with $\int gd\mu_{\mathbb{T}}=0$ and $h_s$ denotes the solution of the master equation \eqref{ME} given $g$, then there exists $c>0$ such that, for all $s\in [0,1]$ and $k\in\Z_{\neq 0}$,
    $$|\hat{h}_{s,k}|\,\leq \frac{c\|g\|_{m,\alpha}}{|k|^{m+\alpha}}.$$
\end{prop}
\begin{proof}
    We have to show that there exists $c_1>0$ such that the coefficients $h_{s,k}$ of $h_s$ in the expansion $h_s=\sum_{|k|\geq 1} h_{s,k} \psi_k$, which a priori only holds a.e. on $\mathbb{T}$, decay as $|h_{s,k}|\leq c_1/|k|^{m+\alpha-\frac{1}{2}}$ for $|k|\geq 1$. 
    We know that there exists $c_2>0$ such that the Fourier coefficients $\hat{g}_k$ of $g$ decay as $|\hat{g}_k|\leq c_2 \|g\|_{m,\alpha}/|k|^{m+\alpha}$ for $|k|\geq 1$, or equivalently, that there exists $c_3>0$ such that the coefficients $g_{k}$ of $g$ in the expansion $g=\sum_{|k|\geq 1} g_{k} \psi_k$ decay as $|g_{k}|\leq c_3/|k|^{m+\alpha-\frac{1}{2}}$ for $|k|\geq 1$. For notional convenience, denote $K=K^{\gamma,V}_{\mathbb{T}}$. It follows from \eqref{ME_FOUR} that
        $ \frac{1}{\beta} \vec{g}^t = (I+sK) \tilde{D} \vec{h}_s^t, $
    and thus
    \begin{equation} \label{h_s^t_series}
        \vec{h}_s^t = \frac{1}{\beta} \tilde{D}^{-1} \vec{g}^t - s\tilde{D}^{-1}K\tilde{D}\vec{h}_s^t.
    \end{equation}
    By the assumption,
        $$ |(D^{-1} \vec{g}^t)_{k,1}| = 2 |g_{-k}| \leq \frac{c_3\|g\|_{m,\alpha}}{|k|^{m+\alpha-\frac{1}{2}}}. $$
    On the other hand, by Hölder's inequality,
        $$ |(\tilde{D}^{-1}K\tilde{D}\vec{h}_s^t)_{k,1}| \leq \sum_{|l|\geq 1} |(\tilde{D}^{-1}K\tilde{D})_{k,l} h_{s,l}| \leq \|h_s\|_{L^2} \left(\sum_{|l|\geq 1} |(\tilde{D}^{-1}K^{\gamma,V}_{\mathbb{T}}\tilde{D})_{k,l}|^2\right)^{\frac{1}{2}}.  $$
    Since $z_e\in C^{m+3,\alpha}$ and $K_{k,l}=2 |kl|^{\frac{1}{2}}G_{k,l}$, whenever $p+q=m+\alpha+2$, there exists $c_4>0$ such that 
        $$|(\tilde{D}^{-1}K\tilde{D})_{k,l}| = \frac{1}{4}|K_{-k,-l}| \leq \frac{c_4}{|k|^{p-\frac{1}{2}}|l|^{q-\frac{1}{2}}}. $$
    Therefore,
        $$ |(\tilde{D}^{-1}K\tilde{D}\vec{h}_s^t)_{k,1}| \leq \frac{c_4 \|h_s\|_{L^2}}{|k|^{p-\frac{1}{2}}} \left(\sum_{|l|\geq 1} \frac{1}{|l|^{2q-1}}\right)^{\frac{1}{2}}, $$
    and we may take $p=m+\alpha$ and $q=2$. Finally, as $\vec{h}_s^t = \tilde{D}^{-1} \frac{1}{\beta} (I+sK)^{-1} \vec{g}^t$ and $\|(I+sK)^{-1}\|_{\ell^2} \leq 1/(1-\kappa)$, we obtain 
        $$ \|\vec{h}_s\|_{\ell^2} \leq c_5 \|\vec{g}\|_{\ell^2} \leq c_6 \|g\|_{m,\alpha},  $$
    as desired.
\end{proof}
\begin{rem}
    We may further refine the condition on $z_e$ by making use of the identity
        $$ \vec{h}_s^t = \frac{1}{\beta} \tilde{D}^{-1} (I-K^{\gamma,V}_{\mathbb{T}}\tilde{D}) \vec{g}^t + s^2 \tilde{D}^{-1}(K^{\gamma,V}_{\mathbb{T}})^2\tilde{D}\vec{h}_s^t ,$$
    which we obtain after applying \eqref{h_s^t_series} twice.
\end{rem}

\section{Proofs of the main results} \label{proofs}

\subsection{Strategy} \label{proofs_strat}

We will further build upon the method from \cite{CourteautJohansson2025}. One substantial new element will be Proposition \ref{G_lb} where we quantify the difference between the log-kernel and Grunsky kernel explicitly.
\medbreak

Define the interpolating Hamiltonian as 
    $$ \mathcal{H}_n^{[s]}(z_1,\dots,z_n) = (1-s) \mathcal{H}_n^0(z_1,\dots,z_n) + s \mathcal{H}_n^V(z_e(z_1),\dots,z_e(z_n)), $$
where $\mathcal{H}_n^V$ is defined in \eqref{H_n_def}. The interpolating Hamiltonian can be written in terms of the interpolating $\log$-kernel $\mathcal{L}_s$ as follows
\begin{align} \label{H_n^s_alt}
    \mathcal{H}_n^{[s]}(z_1,\dots,z_n) = \sum_{1\leq k\neq l\leq n} \mathcal{L}_s(z_k,z_l) + s \sum_{k=1}^n (Vz_e)(z_k).
\end{align}
Particularly, we can also write this in terms of the Grunsky kernel by making use of \eqref{L_s:G}. We will study linear statistics $\sum_{k=1}^n g(z_k)$ w.r.t. the interpolating gas
\[ \frac{1}{n! Z_{n,\beta}^{[s]}} \exp\left(-\frac{\beta}{2} \mathcal{H}_n^{[s]}(z_1,\dots,z_n) + s \sum_{k=1}^n \log |z_e^{(1)}(z_k)|\right) \prod_{k=1}^n |dz_k|,\]
on an appropriate contour $\tau$, where
\begin{equation} \label{s_Z}
    Z_{n,\beta}^{[s]} = \frac{1}{n!} \int_{\tau^n} \exp\left(-\frac{\beta}{2} \mathcal{H}_n^{[s]}(z_1,\dots,z_n) + s \sum_{k=1}^n \log |z_e^{(1)}(z_k)|\right) \prod_{k=1}^n |dz_k|.
\end{equation}
Using the integration rule
$$ \int_\gamma f(z) |dz| = \int_0^1 (fz_e\tau)(t) |(z_e\tau)'(t)| dt = \int_{\tau} (fz_e)(z) |z_e^{(1)}(z)| |dz|, $$
we recover
    $$ Z_{n,\beta}^{[0]} = Z_{n,\beta}^{\tau,0},\quad Z_{n,\beta}^{[1]} = Z_{n,\beta}^{\gamma,V}. $$
Moreover, if denote the expectation with respect to interpolating gas by $\mathbb{E}_{n,\beta}^{[s]}$, then
$$ \mathbb{E}_{n,\beta}^{[0]}\left[e^{\sum_{k=1}^n g(z_k)}\right] =  \mathbb{E}_{n,\beta}^{\tau,0}\left[e^{\sum_{k=1}^n g(z_k)}\right],\quad \mathbb{E}_{n,\beta}^{[1]}\left[e^{\sum_{k=1}^n (gz_e)(z_k)}\right] =  \mathbb{E}_{n,\beta}^{\gamma,V}\left[e^{\sum_{k=1}^n g(z_k)}\right].  $$

In order to describe the asymptotics of $\mathbb{E}_{n,\beta}^{[s]}\left[e^{\sum_{k=1}^n g(z_k)}\right]$, it will be more convenient to parametrize the Coulomb gas by $[a_\tau,b_\tau]^n$ and to examine
    $$  \frac{1}{n! Z_{n,\beta}^{[s]}} \int_{[a_\tau,b_\tau]^n} \exp\left(-\frac{\beta}{2} \widetilde{\mathcal{H}}_n^{[s]}(t_1,\dots,t_n) + \sum_{k=1}^n (s\log |\widetilde{z_e}'(t_k)|+\widetilde{g}(t_k))\right) \prod_{k=1}^n dt_k,$$
where    
$$\widetilde{\mathcal{H}}_n^{[s]}(t_1,\dots,t_n) = \mathcal{H}_n^{[s]}(\tau(t_1),\dots,\tau(t_n)),\quad  \widetilde{z_e}'(t) = z_e^{(1)}(\tau(t)),\quad \widetilde{g}(t) = g(\tau(t)).$$
This tilde-notation will appear often in the remainder of this subsection. Generally, for a function $f:\tau^n\to\R$, we will denote 
$$\tilde{f}(t_1,\dots,t_n) = \tilde{f}(\tau(t_1),\dots,\tau(t_n)),\quad (t_1,\dots,t_n)\in[a_\tau,b_\tau],$$
where, with slight abuse of notation, $\tau:[a_\tau,b_\tau]\to\tau$ denotes the arc-length parametrization of $\tau$. 

Let $\widetilde{\mathbb{E}}_{n,\beta}^{[s]}[\,\cdot\,]$ denote the expectation with respect to the Coulomb gas parametrized by $[a_\tau,b_\tau]^n$. Evidently,
\begin{equation} \label{LS_tilde}
    \mathbb{E}_{n,\beta}^{[s]}\left[e^{\sum_{k=1}^n g(z_k)}\right] = \widetilde{\mathbb{E}}_{n,\beta}^{[s]}\left[e^{\sum_{k=1}^n \tilde{g}(t_k)}\right].
\end{equation}

We can now perform a change of variables of the form $t_k\mapsto t_k + \eps \widetilde{h}(t_k)$ for some $\eps\geq 0$ and $\widetilde{h}=h\circ\tau$ with $h:\tau\to\R$ that will be specialized later. After such a change of variables, the initial expectation can also be written as
\begin{equation} \label{LS_V_s}
    \mathbb{E}_{n,\beta}^{[s]}\left[e^{\sum_{k=1}^n g(z_k)}\right] = \widetilde{\mathbb{E}}_{n,\beta}^{[s]}\left[e^{\widetilde{\mathcal{V}}_{g,h,\eps}^{[s]}(t)}\right],
\end{equation}
in terms of the discrepancy
\begin{align} \label{V_ini}
    \widetilde{\mathcal{V}}_{g,h,\eps}^{[s]}(t) = &-\frac{\beta}{2} (\widetilde{\mathcal{H}}_n^{[s]}(t_1+\eps \widetilde{h}(t_1),\dots,t_n+\eps \widetilde{h}(t_n))-\widetilde{\mathcal{H}}_n^{[s]}(t)) \nonumber \\
    &+ \sum_{k=1}^n\left[ s(\log |\widetilde{z_e}'(t_k+\eps\widetilde{h}(t_k))|-\log |\widetilde{z_e}'(t_k)|) + \widetilde{g}(t_k+\eps\widetilde{h}(t_k)) \right] \nonumber \\
    &+ \sum_{k=1}^n\log(1+\eps\widetilde{h}'(t_k)).
\end{align}
Without loss of generality, we will now assume that 
    $$\int g d\mu_\tau = 0.$$
    
We aim to show that $\widetilde{\mathcal{V}}_{g,h,\frac{1}{n}}^{[s]}$ admits a limit $\mathcal{V}_{g,h}^{[s]}\in\R$ in an appropriate sense. Together with some technical assumptions, the proposition below then allows us to conclude that
\begin{equation*}
    \lim_{n\to\infty} \mathbb{E}_{n,\beta}^{[s]}\left[e^{\sum_{k=1}^n g(z_k)}\right] = e^{\mathcal{V}_{g,h}^{[s]}}.
\end{equation*}

\begin{prop} \label{LS_lim_s_proof}
    Let $E_{n,K}^{[s]}\subset [a_\tau,b_\tau]^n$ and suppose that there exists $\mathcal{V}_{g,h}^{[s]}\in\R$ and $c_1,c_2,c_3>0$ with $c_2>c_3$ such that, for all $n\in\N$,
    \begin{enumerate}[label=$\roman*)$]
        \item \label{lim_E_V} $\lim_{n\to\infty} \widetilde{\mathbb{E}}_{n,\beta}^{[s]}\left[|\widetilde{\mathcal{V}}_{g,h,\frac{1}{n}}^{[s]}(t)-\mathcal{V}_{g,h}^{[s]}| \mathds{1}_{E_{n,K}^{[s]}}\right] = 0,$
        \item \label{est_V} $\sup_{t\in E_{n,K}^{[s]}} \left| \widetilde{\mathcal{V}}_{g,h,\frac{1}{n}}^{[s]}(t) - \mathcal{V}_{g,h}^{[s]} \right| \leq c_1,$ 
        \item \label{est_D} $\widetilde{\mathbb{E}}_{n,\beta}^{[s]}\left[\mathds{1}_{(E_{n,K}^{[s]})^c}(t)\right] \leq e^{-c_2n},$
        \item \label{est_V_glob} $\sup_{t\in [a_\tau,b_\tau]^n} | \widetilde{\mathcal{V}}_{g,h,\frac{1}{n}}^{[s]}(t)| \leq c_3n$.
    \end{enumerate}
    Then,
    \begin{equation} \label{LS_lim_s}
    \lim_{n\to\infty} \mathbb{E}_{n,\beta}^{[s]}\left[e^{\sum_{k=1}^n g(z_k)}\right] = e^{\mathcal{V}_{g,h}^{[s]}}.
\end{equation}
\end{prop}
\begin{proof}
We will consider the decomposition  
    $$ \mathbb{E}_{n,\beta}^{[s]}\left[e^{\sum_{k=1}^n g(z_k)}\right] = \widetilde{\mathbb{E}}_{n,\beta}^{[s]}\left[e^{\widetilde{\mathcal{V}}_{g,h,\frac{1}{n}}^{[s]}(t)} \mathds{1}_{E_{n,K}^{[s]}}(t)\right] + \widetilde{\mathbb{E}}_{n,\beta}^{[s]}\left[e^{\widetilde{\mathcal{V}}_{g,h,\frac{1}{n}}^{[s]}(t)}\mathds{1}_{(E_{n,K}^{[s]})^c}(t)\right]. $$
It follows from \ref{est_D} and \ref{est_V_glob} that
$$ \widetilde{\mathbb{E}}_{n,\beta}^{[s]}\left[e^{\widetilde{\mathcal{V}}_{g,h,\frac{1}{n}}^{[s]}(t)}\mathds{1}_{(E_{n,K}^{[s]})^c}(t)\right] \leq e^{(c_3-c_2)n}. $$
Hence,
\begin{equation} \label{LS_s_ess_dom}
    \mathbb{E}_{n,\beta}^{[s]}\left[e^{\sum_{k=1}^n g(z_k)}\right] = \widetilde{\mathbb{E}}_{n,\beta}^{[s]}\left[e^{\widetilde{\mathcal{V}}_{g,h,\frac{1}{n}}^{[s]}(t)}\mathds{1}_{E_{n,K}^{[s]}}(t)\right] + o(1),\quad n\to\infty.
\end{equation}
On the other hand, assumptions \ref{est_V} and \ref{est_D} yield
\begin{align*}
    \left|\widetilde{\mathbb{E}}_{n,\beta}^{[s]}\left[e^{\widetilde{\mathcal{V}}_{g,h,\frac{1}{n}}^{[s]}(t)-\mathcal{V}_{g,h}^{[s]}}\mathds{1}_{E_{n,K}^{[s]}}(t)\right]-1\right| &\leq \widetilde{\mathbb{E}}_{n,\beta}^{[s]}\left[|e^{\widetilde{\mathcal{V}}_{g,h,\frac{1}{n}}^{[s]}(t)-\mathcal{V}_{g,h}^{[s]}}-1|\mathds{1}_{E_{n,K}^{[s]}}(t)\right] + \widetilde{\mathbb{E}}_{n,\beta}^{[s]}\left[\mathds{1}_{(E_{n,K}^{[s]})^c}(t)\right] \\
    & \leq \widetilde{\mathbb{E}}_{n,\beta}^{[s]}\left[|\widetilde{\mathcal{V}}_{g,h,\frac{1}{n}}^{[s]}(t)-\mathcal{V}_{g,h}^{[s]}|e^{|\widetilde{\mathcal{V}}_{g,h,\frac{1}{n}}^{[s]}(t)-\mathcal{V}_{g,h}^{[s]}|}\mathds{1}_{E_{n,K}^{[s]}}(t)\right] + e^{-c_2n} \\
    & \leq \widetilde{\mathbb{E}}_{n,\beta}^{[s]}\left[|\widetilde{\mathcal{V}}_{g,h,\frac{1}{n}}^{[s]}(t)-\mathcal{V}_{g,h}^{[s]}|\mathds{1}_{E_{n,K}^{[s]}}(t)\right] e^{c_1} + e^{-c_2n}.
\end{align*}
It then remains to use \ref{lim_E_V}.
\end{proof}

Based on \eqref{LS_s_ess_dom}, we will call
$$D_{n,K}^{[s]} = \{ (\tau(t_1),\dots,\tau(t_n))\in\tau^n : (t_1,\dots,t_n)\in E_{n,K}^{[s]}\}.$$
the essential domain of integration of $\mathbb{E}_{n,\beta}^{[s]}\left[e^{\sum_{k=1}^n g(z_k)}\right]$. Proposition \ref{LS_lim_s_proof}~\ref{est_V} indicates how small the set $D_{n,K}^{[s]}$ should be, as a priori we will only be able to show that Proposition \ref{LS_lim_s_proof}~\ref{est_V_glob} holds. However, as indicated by Proposition \ref{LS_lim_s_proof}~\ref{est_D} it shouldn't be too small, as e.g. for $E_{n,K}^{[s]}=\emptyset$, it is not satisfied. 
\medbreak

In Section \ref{S_der_V}, we will explain why the optimal choice for $h$ is $h=h_s$ with $h_s:\tau\to\R$ the solution of the master equation \eqref{ME} given $g:\tau\to\R$. A natural choice for $\widetilde{\mathcal{V}}_{g,h}^{[s]}\in\R$ in Proposition~\ref{lim_E_V} is then
\begin{equation} \label{V_form}
    \widetilde{\mathcal{V}}_{g,h}^{[s]} = \frac{1}{2} \int h_s g^{(1)} d\mu_{\tau} + (1-\frac{\beta}{2}) \int \left(h_s^{(1)} + s h_s (\log |z_e^{(1)}|)^{(1)}\right) d\mu_{\tau}.
\end{equation}
In Section \ref{S_EDI}, we will verify that Proposition \ref{LS_lim_s_proof}~\ref{est_V}--\ref{est_V_glob} are satisfied. To this end, it will be convenient to use the unit circle $\mathbb{T}$ as the reference object $\tau$ as a lot of helpful results are available there. Doing so, we will be able to prove the following theorem.

\begin{thm} \label{av_s}
Let $\gamma$ be a $C^{7,\alpha}$ Jordan curve and suppose that $V\in C^{7,\alpha}(\gamma)$ satisfies Assumption \ref{ass1}. Let $g\in C^{4,\alpha}(\mathbb{T})$ with $\int g d\mu_{\mathbb{T}}=0$ and let $h_s:\mathbb{T}\to\R$ be the solution of the master equation
\eqref{ME} given $g$. Then,
\begin{align} \label{int_ls}
    \lim_{n\to\infty} \log \mathbb{E}_{n,\beta}^{[s]}\left[e^{\sum_{k=1}^n g(z_k)}\right] = \frac{1}{2} \int h_s g^{(1)} d\mu_{\mathbb{T}} + (1-\frac{\beta}{2}) s \int h_s   (\log |z_e^{(1)}|)^{(1)} d\mu_{\mathbb{T}}.
\end{align}
\end{thm}

This partial result will be essential in the proof of Theorem \ref{PF} \& \ref{ThmSz} in Section \ref{proofs_mr}. Another important element in these proofs are the different representations of the terms on the right hand side of \eqref{int_ls} in terms of the Fourier basis and Dirichlet inner product, which will be proven in Section \ref{S_RV}.

\subsection{Proof of Theorem \ref{av_s}}

\subsubsection{Derivation of the limit} \label{S_der_V}

The interplay between the precise change of variables and the structure of the interpolating $\log$-kernel \eqref{L_s:G} allows us to slightly simplify the first term of \eqref{V_ini}. To do so, we note that
\begin{align*}
    &\widetilde{\mathcal{L}}_0(u+\eps\widetilde{h}(u),v+\eps\widetilde{h}(v))-\widetilde{\mathcal{L}}_0(u,v) = \log \left| \frac{\tau(u)-\tau(v)}{\tau(u+\eps\widetilde{h}(u))-\tau(v+\eps\widetilde{h}(v))} \right|, \\
    &\widetilde{\mathcal{G}}(u,v) =  \log\left|\frac{\tau(u)-\tau(v)}{\widetilde{z_e}(u)-\widetilde{z_e}(v)}\right| + \frac{1}{2} (V\widetilde{z_e}(u)+V\widetilde{z_e}(v)),
\end{align*}
can be extended to the diagonal as
\begin{align*}
    &\lim_{v\to u,v\in\tau} (\widetilde{\mathcal{L}}_0(u+\eps \widetilde{h}(u),v+\eps \widetilde{h}(v))-\widetilde{\mathcal{L}}_s(u,v)) = - \log|1+\eps \widetilde{h}'(u)|, \\
    &\lim_{v\to u,v\in\tau} \widetilde{\mathcal{G}}(u,v) = - \log|\widetilde{z_e}'(u)|+(V\widetilde{z_e})(u),
\end{align*}
using the fact that $|\tau'|=1$. In that case, as
    $$ \widetilde{\mathcal{H}}_n^{[s]}(t) = \sum_{1\leq k\neq l\leq n} \widetilde{\mathcal{L}}_s(t_k,t_l) + s \sum_{k=1}^n (V\widetilde{z_e})(t_k),$$
we obtain
\begin{align*}
    &\widetilde{\mathcal{H}}_n^{[s]}(t_1+\eps \widetilde{h}(t_1),\dots,t_n+\eps \widetilde{h}(t_n))-\widetilde{\mathcal{H}}_n^{[s]}(t)  \\
    &= \sum_{k,l=1}^n \left(\widetilde{\mathcal{L}}_s(t_k+\eps\widetilde{h}(t_k),t_l+\eps\widetilde{h}(t_l))-\widetilde{\mathcal{L}}_s(t_k,t_l)\right) \\
    &\quad + \sum_{k=1}^n \left(\log|1+\eps \widetilde{h}'(t_k)|+s(\log|\widetilde{z_e}'(t_k+\eps \widetilde{h}(t_k))|-\log|\widetilde{z_e}'(t_k)|)\right).
\end{align*}
As a consequence, it follows from \eqref{V_ini} that
$$ \widetilde{\mathcal{V}}_{g,h,\eps}(t) = \widetilde{\mathcal{W}}_1(\eps;t) + \widetilde{\mathcal{W}}_2(\eps;t), $$
with
\begin{align*}
    \widetilde{\mathcal{W}}_1(\eps;t) &= - \frac{\beta}{2} \sum_{k,l=1}^n \left(\widetilde{\mathcal{L}}_s(t_k+\eps \widetilde{h}(t_k),t_l+\eps \widetilde{h}(t_l)) - \widetilde{\mathcal{L}}_s(t_k,t_l)\right), \\
    \widetilde{\mathcal{W}}_2(\eps;t) &= (1-\frac{\beta}{2}) \sum_{k=1}^n \left(\log(1+\eps \widetilde{h}'(t_k)) + s (\log |\widetilde{z_e}'(t_k+\eps \widetilde{h}(t_k))|-\log |\widetilde{z_e}'(t_k)|) \right) \\
    &\quad + \sum_{k=1}^n \widetilde{g}(t_k+\eps \widetilde{h}(t_k)) 
 \end{align*}
As such $\widetilde{\mathcal{V}}_{g,h,\eps}(t)$ can be represented as a sum of a double sum and a single sum. We will now apply Taylor's theorem to write 
\begin{align*}
    \widetilde{\mathcal{W}}_1(\eps) &= \widetilde{\mathcal{W}}_1'(0)\eps + \widetilde{\mathcal{W}}_1''(0) \frac{\eps^2}{2} + \frac{\eps^3}{2} \int_0^1 \widetilde{\mathcal{W}}_1'''(a\eps) (1-a)^2 da, \\
    \widetilde{\mathcal{W}}_2(\eps) &=  \widetilde{\mathcal{W}}_2(0) + \widetilde{\mathcal{W}}_2'(0)\eps +  \eps^2 \int_0^1 \widetilde{\mathcal{W}}_2''(a\eps) (1-a) da.
\end{align*}

Take $\eps=1/n$. We will proceed under the assumption that, uniformly for $t\in [a_\tau,b_\tau]^n$ as $n\to\infty$,
\begin{align}
    \widetilde{\mathcal{W}}_1(\eps) &= \frac{1}{n}\widetilde{\mathcal{W}}_1'(0) + \frac{1}{2n^2}\widetilde{\mathcal{W}}_1''(0) + o(1), \label{W_1_ass} \\
    \widetilde{\mathcal{W}}_2(\eps) &=  \widetilde{\mathcal{W}}_2(0) + \frac{1}{n} \widetilde{\mathcal{W}}_2'(0) + o(1). \label{W_2_ass}
\end{align}

In a sufficiently regular framework on the unit circle $\tau=\mathbb{T}$, this can definitely be guaranteed as shown in the following result.

\begin{prop} \label{Taylor_err}
    Suppose that $\gamma$ is a $C^{4,\alpha}$ Jordan curve and that $V\in C^{4,\alpha}(\gamma)$ satisfies Assumption \ref{ass1}. If $g\in C^{2,\alpha}(\mathbb{T})$ and
    $h\in C^{1,\alpha}(\mathbb{T})$, then there exist $n_0(\|h\|_{1,\alpha})\in\N$ and $C(\|g\|_{2,\alpha},\|h\|_{1,\alpha})>0$ such that for all $n\geq n_0$,
    \begin{align*}
        &\sup_{s\in[0,1]}
        \sup_{t^{(n)}\in[0,2\pi]^n}
        \sup_{a\in[0,1]}
        \left|
        \frac{1}{n^2}
        \widetilde{\mathcal{W}}_1'''
        \left(\frac{a}{n};t^{(n)}\right)
        \right|
        \leq
        C(\|g\|_{2,\alpha},\|h\|_{1,\alpha}),
        \\
        &\sup_{s\in[0,1]}
        \sup_{t^{(n)}\in[0,2\pi]^n}
        \sup_{a\in[0,1]}
        \left|
        \frac{1}{n}
        \widetilde{\mathcal{W}}_2''
        \left(\frac{a}{n};t^{(n)}\right)
        \right|
        \leq
        C(\|g\|_{2,\alpha},\|h\|_{1,\alpha}).
    \end{align*}
\end{prop}
\begin{proof}
Under the stated conditions, we know from Proposition \ref{ep_reg} that $z_e\in C^{4,\alpha}(\mathbb{T})$ and $\inf_{\mathbb{T}}|z_e^{(1)}|\,>0$. Denote
$ T_\eps(t)=t+\eps\widetilde{h}(t)$ and suppose that $n\in\N$ is sufficiently large so that
    $$
    \frac{1}{2}
    \leq T_\eps'(t)=1+\eps\widetilde{h}'(t)
    \leq\frac{3}{2}.
    $$
We will first focus on $\widetilde{\mathcal{W}}_1$. Recall that $\widetilde{\mathcal{L}}_s
=\widetilde{\mathcal{L}}_0+s\widetilde{\mathcal{G}}$. For the contribution of $\widetilde{\mathcal{L}}_0$, suppose that $u-v\not\in 2\pi\Z$ and define
\begin{align*}
    Q_\eps(u,v)
    =
    \widetilde{\mathcal{L}}_0(T_\eps(u),T_\eps(v))
    -\widetilde{\mathcal{L}}_0(u,v)=
    \log\left|
    \frac{\sin\big(\frac{u-v}{2}\big)}
    {\sin\big(\frac{u-v+\eps(\widetilde{h}(u)
                     -\widetilde{h}(v))}{2}\big)}
    \right|.
\end{align*}
Its third derivative to $\eps$ can be written as
    $$
    \partial_\eps^3 Q_\eps(u,v)
    =
    -\frac{1}{4}
    \frac{\cos\big(\frac{T_\eps(u)-T_\eps(v)}{2}\big)}
         {\sin^3\big(\frac{T_\eps(u)-T_\eps(v)}{2}\big)} (\widetilde{h}(u)-\widetilde{h}(v))^3.
    $$
It will be convenient to introduce $ d(u,v)=\min_{j\in\Z}|u-v+2\pi j|$. The estimates
\begin{align*}
    \left|
    \sin\left(\frac{T_\eps(u)-T_\eps(v)}{2}\right)
    \right|
    =
    \sin\left(
    \frac{d(T_\eps(u),T_\eps(v))}{2}
    \right)
    \geq
    \frac{1}{\pi}d(T_\eps(u),T_\eps(v))
    \geq \frac{1}{2\pi}d(u,v),
\end{align*}
then show that
    $$
    |\partial_\eps^3 Q_\eps(u,v)|
    \leq 2\pi^3\|h\|_{1,\alpha}^3.
    $$
On the diagonal, the continuous extension of $Q_\eps$ is
    $$
    Q_\eps(u,u)
    =-\log(1+\eps\widetilde{h}'(u)),
    $$
and its third derivative to $\eps$ is
    $$
    \partial_\eps^3 Q_\eps(u,u)
    =
    -\frac{2\widetilde{h}'(u)^3}
    {(1+\eps\widetilde{h}'(u))^3}.
    $$
This agrees with the limit of the off-diagonal expression and satisfies the same uniform bound. For the contribution of $\widetilde{\mathcal{G}}$, we first observe that $\widetilde{\mathcal{G}}\in C^{3,\alpha}$ as $z_e\in C^{4,\alpha}(\mathbb{T})$, $V\in C^{3,\alpha}(\gamma)$ and $\inf_{\mathbb{T}}|z_e^{(1)}|\,>0$. As a consequence, there exists $C_1>0$ such that
    $$
    \left|
    \partial_\eps^3
    \big[
    \widetilde{\mathcal{G}}(T_\eps(u),T_\eps(v))
    \big]
    \right|
    \leq C_1\|h\|_\infty^3.
    $$
We can therefore conclude that there exists $C_2>0$ such that for all $s\in[0,1]$, $t^{(n)}\in[0,2\pi]^n$ and $0\leq\eps\leq1/n_0$,
\begin{align*}
    \left|
    \widetilde{\mathcal{W}}_1'''(\eps;t^{(n)})
    \right|
    \leq
    \frac{\beta}{2}\sum_{k,l=1}^n
    \left(
    2\pi^3\|h\|_{1,\alpha}^3+C_1\|h\|_\infty^3
    \right)
    \leq n^2 C_2(\|h\|_{1,\alpha}).
\end{align*}
We will now consider $\widetilde{\mathcal{W}}_2(\eps;t^{(n)})$. Its second derivative to $\eps$ is
\begin{align*}
    \widetilde{\mathcal{W}}_2''(\eps;t^{(n)})
    =
    &\sum_{k=1}^n
    \widetilde{h}(t_k)^2
    \widetilde{g}''(T_\eps(t_k)) \\
    &+
    (1-\frac{\beta}{2})
    \sum_{k=1}^n
    \left(
    s\widetilde{h}(t_k)^2(\log|\widetilde{z_e}'|)''(T_\eps(t_k))
    -
    \frac{\widetilde{h}'(t_k)^2}
    {(1+\eps\widetilde{h}'(t_k))^2}
    \right).
\end{align*}
It then follows from $1+\eps\widetilde{h}'=T_\eps'\geq 1/2$ that there exists 
$C_3(\|g\|_{2,\alpha},\|h\|_{1,\alpha})>0$ such that for all $s\in[0,1]$, $t^{(n)}\in[0,2\pi]^n$ and $0\leq\eps\leq1/n_0$,
    $$
    \left|
    \widetilde{\mathcal{W}}_2''(\eps;t^{(n)})
    \right|
    \leq
    n C_3(\|g\|_{2,\alpha},\|h\|_{1,\alpha}),
    $$
Taking $\eps=a/n$ with $n\geq n_0$ in both estimates then finishes the proof.
\end{proof}

At this point, we can move back to the complex variable using \eqref{tan_der}. Evidently,
    $$ \mathcal{W}_2(0) = \sum_{k=1}^n g(z_k). $$
Defining
    $$ m = (1-\frac{\beta}{2}) (h^{(1)} + s h (\log|z_e^{(1)}|)^{(1)}),$$
gives
$$\mathcal{W}_2'(0) = \sum_{k=1}^n (h(z_k) g^{(1)}(z_k) + m(z_k)).$$
Lastly,
$$\mathcal{W}_1^{(j)}(0) = - \frac{\beta}{2} \sum_{k,l=1}^n [h(z_k) \partial_1 + h(z_l)\partial_2]^{j} \mathcal{L}_s(z_k,z_l), \quad j\in\{1,2\}.$$
Given $z=(z_1,\dots,z_n)\in \tau^n$, we will now write each of these contributions in terms of the empirical measure
    $$ \mu_{z^{(n)}} = \frac{1}{n} \sum_{k=1}^n \delta_{z_k^{(n)}}. $$
For $\mathcal{W}_1''(0)$ and $\mathcal{W}_2'(0)$, this gives 
\begin{align*}
    \mathcal{W}_1''(0) = &- \frac{\beta}{2} n^2 \iint [h(u) \partial_u +  h(v)\partial_v]^2 \mathcal{L}_s(u,v) d\mu_{z^{(n)}}(u)d\mu_{z^{(n)}}(v), \\
    \mathcal{W}_2'(0) =  & n \int (h g^{(1)} + m) d\mu_{z^{(n)}}.
\end{align*}
For the other contributions, we will further compare $\mu_{z^{(n)}}$ with the equilibrium measure $\mu_\tau$ on $\tau$ via the signed measure
    $$ \lambda_{z^{(n)}} = n (\mu_{z^{(n)}} - \mu_\tau),$$
which measures the fluctuations of $\mu_{z^{(n)}}$ around $\mu_\tau$. Since 
$$d\mu_{z^{(n)}}=d\mu_{\tau}+\frac{1}{n} d\lambda_{z^{(n)}},$$
and $\int g d\mu_\tau =0$, for $\mathcal{W}_2(0)$, this gives
$$\mathcal{W}_2(0) = n \int g d\mu_{z^{(n)}} = \int g d\lambda_{z^{(n)}}.$$
Since 
\begin{align*}
    d\mu_{z^{(n)}}(u)d\mu_{z^{(n)}}(v) &= d\mu_{\tau}(u)d\mu_{\tau}(v) + \frac{1}{n}(d\lambda_{z^{(n)}}(u)d\mu_{\tau}(v)+d\mu_{\tau}(u)d\lambda_{z^{(n)}}(v)) \\
    &\quad + \frac{1}{n^2} d\lambda_{z^{(n)}}(u)d\lambda_{z^{(n)}}(v),
\end{align*}    
for $\mathcal{W}_1'(0)$, this gives
\begin{align*}
    \mathcal{W}_1'(0) = &- \frac{\beta}{2} n^2 \iint [h(u) \partial_u +  h(v)\partial_v] \mathcal{L}_s(u,v) d\mu_{\tau}(u)d\mu_{\tau}(v) \\
    &- \beta n \iint [h(u) \partial_u +  h(v)\partial_v] \mathcal{L}_s(u,v) d\mu_{\tau}(u) d\lambda_{z^{(n)}}(v) \\
    & - \frac{\beta}{2} \iint [h(u) \partial_u +  h(v)\partial_v] \mathcal{L}_s(u,v) d\lambda_{z^{(n)}}(u)d\lambda_{z^{(n)}}(v).
\end{align*}
The fact that $\mathcal{L}_s$ admits an expansion of the form
    $$ \mathcal{L}_s(u,v) = c_s + \phi(u) L_s \phi(v)^t,$$
see \eqref{L_s_exp}, the latter can be simplified substantially, namely
\begin{align*}
    \mathcal{W}_1'(0) = &- \beta n \iint h(u)\mathcal{L}_s^{(1,0)}(u,v) d\mu_{\tau}(u) d\lambda_{z^{(n)}}(v), \\
    & - \frac{\beta}{2} \iint [h(u) \partial_u +  h(v)\partial_v] \mathcal{L}_s(u,v) d\lambda_{z^{(n)}}(u)d\lambda_{z^{(n)}}(v).
\end{align*}

In order to cancel some contributions of $d\lambda_{z^{(n)}}$, it is then convenient to take $h=h_s$ with $h_s$ such that 
    $$ g(v) = \beta \int h_s(u)\mathcal{L}_s^{(1,0)}(u,v) d\mu_{\tau}(u).  $$
We will denote $m_s=m$ for this choice of $h$. In that case, uniformly for $z^{(n)}\in\tau^n$ as $n\to\infty$,
\begin{align*}
    \mathcal{V}_{g,h_s,\frac{1}{n}}(z^{(n)}) = &- \frac{\beta}{2n} \iint [h_s(u) \partial_u +  h_s(v)\partial_v] \mathcal{L}_s(u,v) d\lambda_{z^{(n)}}(u)d\lambda_{z^{(n)}}(v) \\
     & - \frac{\beta}{4} \iint [h_s(u) \partial_u +  h_s(v)\partial_v]^2 \mathcal{L}_s(u,v) d\mu_{z^{(n)}}(u)d\mu_{z^{(n)}}(v) \\
     &+ \int  (h_s g^{(1)} + m_s) d\mu_{z^{(n)}} + o(1).
\end{align*}

As a consequence, whenever $z_e\in C^{4,\alpha}(\tau)$, $V\in C^{3,\alpha}(\gamma) $, $g\in C^{2,\alpha}(\gamma)$ and $h_s\in C^{1,\alpha}(\gamma)$, there exists $c_3(\|g\|_{2,\alpha},\|h_s\|_{1,\alpha})>0$ such that
\begin{equation} \label{est_V_glob_uni}
    \sup_{z^{(n)}\in \tau^n} | \widetilde{\mathcal{V}}_{g,h,\frac{1}{n}}^{[s]}(z^{(n)})| \leq c_3(\|g\|_{2,\alpha},\|h_s\|_{1,\alpha}) n.
\end{equation}
This already shows that Proposition \ref{LS_lim_s_proof}~\ref{est_V_glob} is satisfied under these regularity assumptions.
\medbreak

We will proceed under the assumption that the essential domain of integration $D_{n,K}^{[s]}$ is such that
\begin{equation} \label{EDI_WC}
    \lim_{n\to\infty} \sup_{z^{(n)}\in D_{n,K}^{[s]}} \left| \int f d\mu_{z^{(n)}} - \int f d\mu_\tau \right| = 0.
\end{equation}
We will show that in a sufficiently regular framework on the unit circle $\tau=\mathbb{T}$, this can definitely be guaranteed.
\medbreak

Under this assumption, uniformly for $z^{(n)}\in D_{n,K}^{[s]}$ as $n\to\infty$,
\begin{align*}
    \mathcal{V}_{g,h_s,\frac{1}{n}}(z^{(n)}) = &- \frac{\beta}{2n} \iint [h_s(u) \partial_u +  h_s(v)\partial_v] \mathcal{L}_s(u,v) d\lambda_{z^{(n)}}(u)d\lambda_{z^{(n)}}(v) \\
     & - \frac{\beta}{4} \iint [h_s(u) \partial_u +  h_s(v)\partial_v]^2 \mathcal{L}_s(u,v) d\mu_{\tau}(u)d\mu_{\tau}(v) \\
     &+ \int  (h_s g^{(1)} + m_s) d\mu_{\tau} + o(1).
\end{align*}
Using the expansion \eqref{L_s_exp} of $\mathcal{L}_s$ and the master equation, the second term can be further simplified to
$$- \frac{1}{2} \int h_s g^{(1)} d\mu_{\tau}. $$
Consequently, uniformly for $z^{(n)}\in D_{n,K}^{[s]}$ as $n\to\infty$,
\begin{align} \label{V_form}
    \mathcal{V}_{g,h_s,\frac{1}{n}}(z^{(n)}) = & - \frac{\beta}{2n} \iint [h_s(u) \partial_u +  h_s(v)\partial_v] \mathcal{L}_s(u,v) d\lambda_{z^{(n)}}(u)d\lambda_{z^{(n)}}(v) \nonumber \\
    &+ \int  (\frac{1}{2}h_s g^{(1)} + m_s) d\mu_{\tau} + o(1),
\end{align}
and thus, uniformly for $z^{(n)}\in D_{n,K}^{[s]}$ as $n\to\infty$,
    $$ \mathcal{V}_{g,h_s,\frac{1}{n}}(z^{(n)}) = \mathcal{V}_{g,h_s} - \frac{\beta}{2n} \iint [h_s(u) \partial_u +  h_s(v)\partial_v] \mathcal{L}_s(u,v) d\lambda_{z^{(n)}}(u)d\lambda_{z^{(n)}}(v) + o(1). $$
We therefore conclude that Proposition~\ref{lim_E_V} is satisfies whenever \eqref{W_1_ass}, \eqref{W_2_ass}, \eqref{EDI_WC} and
\begin{equation} \label{err_0}
    \lim_{n\to\infty} \mathbb{E}_{n,\beta}^{[s]}\left[\frac{1}{n}\left|\iint [h_s(u) \partial_u +  h_s(v)\partial_v] \mathcal{L}_s(u,v) d\lambda_{z^{(n)}}(u)d\lambda_{z^{(n)}}(v)\right| \mathds{1}_{D_{n,K}^{[s]}}(z^{(n)}) \right] = 0, 
\end{equation}
are satisfied. In the next section, we will show that this will be the case in a sufficiently regular framework on the unit circle $\tau=\mathbb{T}$.

\subsubsection{Essential domain of integration} \label{S_EDI}

Denote $[0,2\pi]^n(\uparrow) = \{ \theta\in[0,2\pi]^n : \theta_1\leq \dots\leq \theta_n \}$. We consider the (Fekete) points
    $$ \xi^{(n)}_u = e^{i\alpha_u^{(n)}},\quad \alpha^{(n)}_u =  \frac{2\pi u}{n},\quad u\in\{0,\dots,n-1\}. $$
Let $K\in\R$ and define the sets
\begin{align*}
    E_{n,K}^{[s]} &:= \{ \theta\in[0,2\pi]^n: \sum_{1\leq u\neq v \leq n} (\mathcal{L}_s(e^{i\theta_u},e^{i\theta_v})-\mathcal{L}_0(\xi^{(n)}_u, \xi^{(n)}_v)-sG_{0,0}) \leq Kn\}, \\
    D_{n,K}^{[s]} &:= \{ (e^{i\theta_1},\dots,e^{i\theta_n})\in\mathbb{T}^n : (\theta_1,\dots,\theta_n)\in E_{n,K}^{[s]}\}.
\end{align*}
We will show that for large enough $K$, these sets will be essential domains of integration for $\mathbb{E}_{n,\beta}^{[s]}$, i.e. \eqref{est_D} is satisfied. In order to show this, we first need some initial control over the partition function. The idea will be to restrict to an even smaller (essential) domain of integration where the integrand is close to its maximum. We expect this maximum to be close to the value achieved at $\xi^{(n)}$. The points $\xi^{(n)}$ have the following key property.

\begin{prop} \label{CR_FEK}
    There exists $C_1>0$ such that for all $f\in C^{1}(\mathbb{T})$ and $n\in\N$,
        $$ \left| \int f d\mu_{\xi^{(n)}} - \int f d\mu_{\mathbb{T}} \right| \leq C_1 \frac{\|f^{(1)}\|_{\infty}}{n}. $$
\end{prop}
\begin{proof}
    This follows from a Riemann sum argument.
\end{proof}

The partition function can now be estimated as follows.

\begin{prop} \label{PF_LB}
Suppose that $z_e\in C^{4,\alpha}(\mathbb{T})$ and $V\in C^{3,\alpha}(\gamma)$. Then there exists $C_3\in\R$ such that for all $s\in [0,1]$ and $n\in\N$,
    $$ Z_{n,\beta}^{[s]} \geq n^{-n} \exp\left(-\frac{\beta}{2} \sum_{1\leq u\neq v \leq n} (\mathcal{L}_0(\xi^{(n)}_u, \xi^{(n)}_v)+sG_{0,0}) + C_3 n\right).  $$
\end{prop}
\begin{proof}
Denote 
    $$ \mathcal{E}_n = \left\{\theta\in[0,2\pi]^n : \max_{1\leq u\leq n} |\theta_u-\alpha^{(n)}_u|  \leq \frac{1}{n} \right\}\subset [0,2\pi]^n(\uparrow), $$
so that
    $$ Z_{n,\beta}^{[s]} \geq \int_{\mathcal{E}_n} \exp\left( -\frac{\beta}{2} \sum_{1\leq u\neq v \leq n} \mathcal{L}_s(e^{i\theta_u},e^{i\theta_v}) \right) \prod_{k=1}^n d\theta_k.  $$
The idea will be to obtain an appropriate upper bound for 
\begin{equation} \label{sum_Ls}
    \sum_{1\leq u\neq v \leq n} \mathcal{L}_s(e^{i\theta_u},e^{i\theta_v}) = \sum_{1\leq u\neq v \leq n} \mathcal{L}_0(e^{i\theta_u},e^{i\theta_v}) + s\sum_{1\leq u\neq v \leq n} \mathcal{G}(e^{i\theta_u},e^{i\theta_v}),
\end{equation}
and to use the fact that
    $$ \int_{\mathcal{E}_n} \prod_{k=1}^n d\theta_k \geq (2n)^{-n}. $$
We first focus on the first term in \eqref{sum_Ls}. Denote $\beta_u(a) = \alpha^{(n)}_u + a t_u(\theta)$ with $t_u(\theta) = \theta_u-\alpha^{(n)}_u - \sigma(\theta)$ and $\sigma(\theta) = \frac{1}{n}\sum_{u=1}^n (\theta_u-\alpha^{(n)}_u) $ and consider 
        $$\p_n(a) = \sum_{1\leq u \neq v \leq n} \widetilde{\mathcal{L}}_0(\beta_u(a),\beta_v(a)). $$
    Then we can write
        $$\p_n(a) = - \sum_{1\leq u \neq v \leq n} \log\left|2\sin\left(\frac{\alpha^{(n)}_u-\alpha^{(n)}_v + a (t_u-t_v)}{2}\right)\right|. $$
    Hence, by Taylor's theorem,
    \begin{equation} \label{L0_Taylor}
        \p_n(1) = \p_n(0) + \p_n'(0) + \int_0^1 \p_n''(a) (1-a) da,
    \end{equation}
    with
    \begin{align*}
        \p_n(1) &= \sum_{1\leq u \neq v \leq n} \widetilde{\mathcal{L}}_0(\theta_u,\theta_v), \quad \p_n(0) = \sum_{1\leq u \neq v \leq n} \widetilde{\mathcal{L}}_0(\alpha^{(n)}_u,\alpha^{(n)}_v), \\
        \p_n'(0) &= - \sum_{1\leq u \neq v \leq n} \cot\left(\frac{\alpha^{(n)}_u-\alpha^{(n)}_v}{2}\right) \frac{t_u-t_v}{2},\quad 
        \p_n''(a) = \sum_{1\leq u \neq v \leq n} \frac{(t_u-t_v)^2}{4 \sin^2(\frac{\beta_u(a)-\beta_v(a)}{2})}.
    \end{align*}
    Due to the identity $\sum_{k=1}^{n-1} \cot\left(\frac{\pi k}{n}\right)=0$, we have
        $$ \p_n'(0) = - \sum_{u=1}^n t_u \sum_{k=1}^{n-1} \cot\left(\frac{\pi k}{n}\right) = 0.$$
    On the other hand, see \cite[Proof of Lem. 3.2]{CourteautJohansson2025},
    $$\p_n''(a) \leq \frac{4}{3}(n-1/n) . $$   
We can therefore conclude that there exists $c_1>0$ such that
    $$\sum_{1\leq u\neq v \leq n} (\mathcal{L}_0(e^{i\theta_u},e^{i\theta_v})-\mathcal{L}_0(\xi^{(n)}_u, \xi^{(n)}_v)) \leq c_1 n.$$
For the second term in \eqref{sum_Ls}, we write
    $$ \sum_{1\leq u\neq v \leq n} \mathcal{G}(e^{i\theta_u},e^{i\theta_v}) = \sum_{u,v=1}^n \mathcal{G}(e^{i\theta_u},e^{i\theta_v}) + \sum_{u=1}^n (\log|z_e^{(1)}(e^{i\theta_u})|+(Vz_e)(e^{i\theta_u})). $$
Since $z_e\in C^1$ and $V\in C^0$, we have
    $$ \sum_{u=1}^n (\log|z_e^{(1)}(e^{i\theta_u})|+(Vz_e)(e^{i\theta_u})) \leq (\log \| z_e^{(1)} \|_{\infty}+\| V \|_{\infty}) n.$$
 On the other hand, 
     $$ \sum_{u,v=1}^n (\mathcal{G}(e^{i\theta_u},e^{i\theta_v}) - G_{0,0}) = \sum_{|k|,|l|\geq 1} G_{k,l} \left( \sum_{u=1}^n \phi_k(\theta_u) \right) \left( \sum_{v=1}^n \phi_l(\theta_v) \right). $$
Since $\theta \in \mathcal{E}_n$, we have the following property: there exists $C_2>0$ such that for all $f\in C^{1}$ and $n\in\N$,
        $$ \left| \int f d\mu_{z^{(n)}} - \int f d\mu_{\mathbb{T}} \right| \leq C_2 \frac{\|f^{(1)}\|_{\infty}}{n}. $$
where $z^{(n)}=(e^{i\theta_1},\dots,e^{i\theta_n})$. Indeed, this follows from Proposition \ref{CR_FEK} and the fact that
    $$ \left| \int f d\mu_{z^{(n)}} - \int f d\mu_{\xi^{(n)}} \right| \leq \frac{\|f^{(1)}\|_{\infty}}{n} \sum_{u=1}^n |t_u| \leq \frac{\|f^{(1)}\|_{\infty}}{n}.$$
Hence, as $\|\phi_k\|_{\infty}\leq \sqrt{2}$ and $\int \phi_l d\mu_{\mathbb{T}}=0$ and $\|\phi_l'\|_{\infty}\leq |l| \sqrt{2}$, we obtain
    $$ \sum_{u,v=1}^n (\mathcal{G}(e^{i\theta_u},e^{i\theta_v}) - G_{0,0}) \leq  2 C_2 n \sum_{|k|,|l|\geq 1} |G_{k,l}| |l|. $$
Hence, whenever $z_e\in C^{4,\alpha}$ and $V\in C^{3,\alpha}$, by Proposition \ref{reg_G}, there exists $c_2>0$ such that
    $$\sum_{u,v=1}^n (\mathcal{G}(e^{i\theta_u},e^{i\theta_v})-G_{0,0}) \leq c_2 n. $$
This finishes the proof.
\end{proof}

We now show that Proposition~\ref{est_D} is satisfied.

\begin{prop} \label{est_D_prop}
    Suppose that $z_e\in C^{4,\alpha}(\mathbb{T})$ and $V\in C^{3,\alpha}(\gamma)$. Then there exists $C_2\in\R$ such that for all $K\in\R$ and $n\in\N$,
    $$\sup_{s\in[0,1]} \mathbb{P}_{n,\beta}^{[s]}[(D_{n,K}^{[s]})^c] \leq \exp(C_2n-\frac{\beta}{2}Kn).$$
\end{prop}
\begin{proof}
    Proposition \ref{PF_LB} implies that
    \begin{align*}
    \mathbb{P}[(D_{n,K}^{[s]})^c]
    &\leq \frac{n^n}{n!} e^{-C_2n} \int_{(E_{n,K}^{[s]})^c}
    \exp\bigg( -\frac{\beta}{2} \sum_{1\leq u\neq v \leq n}
    \big(\mathcal{L}_s(e^{i\theta_u},e^{i\theta_v}) \\
    &\hspace{45mm}-\mathcal{L}_0(\xi^{(n)}_u, \xi^{(n)}_v)-sG_{0,0}\big)\bigg)
    \prod_{k=1}^n d\theta_k.
    \end{align*}
    Hence, together with the definition of $E_{n,K}^{[s]}$, we obtain
    $$\mathbb{P}[(D_{n,K}^{[s]})^c] \leq \frac{n^n}{n!} e^{-C_2n-\frac{\beta}{2}Kn}\int_{[0,2\pi]^n} \prod_{k=1}^n d\theta_k \leq \frac{n^n}{n!} (2\pi)^n e^{-C_2n-\frac{\beta}{2}Kn}.$$
    This proves the desired estimate.
\end{proof}

In order to show that Proposition \ref{LS_lim_s_proof}~\ref{est_V} and \ref{lim_E_V} are satisfied, we will extract some information from the essential domain of integration $D_{n,K}^{[s]}$ about the rate of decay of
    $$\sup_{z^{(n)}\in D_{n,K}^{[s]}} \left| \int f \, d\lambda_{z^{(n)}} \right|,\quad n\to\infty.$$
We will do so by studying the differences
\begin{equation} \label{t_u_def}
    t_u(\theta) = \theta_{(u)}-\alpha^{(n)}_u - \sigma(\theta),
\end{equation}
for $\theta\in E_{n,K}^{[s]}$ where $\theta_{(1)}\leq\dots\leq \theta_{(n)} $ are the ordered coordinates of $(\theta_1,\dots,\theta_n)$ and
$$\sigma(\theta) = \frac{1}{n}\sum_{u=1}^n (\theta_{(u)}-\alpha^{(n)}_u).$$
The idea is that by Proposition \ref{CR_FEK} the measure $\mu_{\xi^{n}}$ is already close to the equilibrium measure $\mu_{\mathbb{T}}$. To proceed, we need to obtain an appropriate lower bound for 
$$\sum_{1\leq u\neq v \leq n} (\mathcal{L}_s(e^{i\theta_u},e^{i\theta_v})-\mathcal{L}_0(\xi^{(n)}_u, \xi^{(n)}_v)-G_{0,0}).$$
in terms of the $t_u(\theta)$. The result below shows that the main contribution to the sum is coming from the initial $\log$-kernel.

\begin{prop} \label{G_lb}
Suppose that $z_e\in C^{4,\alpha}(\mathbb{T})$ and $V\in C^{3,\alpha}(\gamma)$. Let $\kappa\in(0,1)$ be as in Proposition~\ref{SGI_UC}. Then there exists a function $F:\N\to\R$ with $F(n) = O(n)$ as $n\to\infty$ such that for all $z\in\mathbb{T}^n$,
    $$ \sum_{1\leq u \neq v \leq n} (\mathcal{G}(z_u,z_v)-G_{0,0}) \geq - \kappa^{\frac{1}{2}} \sum_{1\leq u \neq v \leq n} (\mathcal{L}_0(z_u,z_v)-\mathcal{L}_0(\xi^{(n)}_u, \xi^{(n)}_v)) + F(n). $$
\end{prop}
\begin{proof}
    After including the diagonal in the double sum on the left hand side, the left hand side can be written as
    $$\sum_{u,v=1}^n \sum_{|k|,|l|\geq 1} (K^{\gamma,V}_{\mathbb{T}})_{k,l} \psi_k(\theta_u) \psi_l(\theta_v) + \sum_{u=1}^n (\log|z_e^{(1)}(e^{i\theta_u})|-(Vz_e)(e^{i\theta_u})-G_{0,0}). $$
    Since $z_e\in C^1$ and $V\in C^0$, there exists $c_1>0$ such that
        $$\left|\sum_{u=1}^n (\log|z_e^{(1)}(e^{i\theta_u})|-(Vz_e)(e^{i\theta_u})-G_{0,0})|\right|\leq n c_1,$$    
    which can be absorbed in $F(n)$. We will use the strengthened Grunsky inequality to estimate the remaining part
        $$\sum_{u,v=1}^n \sum_{|k|,|l|\geq 1} (K^{\gamma,V}_{\mathbb{T}})_{k,l} \psi_k(\theta_u) \psi_l(\theta_v). $$
    Denote $f_k = \sum_{u=1}^n \psi_k(\theta_u)$ and consider $f=\sum_{|k|=1}^N f_k \psi_k $ for some appropriate $N\in\N$ that we will fix later. The first term can then be written as
    $$ \sum_{|k|,|l|\geq 1} (K^{\gamma,V}_{\mathbb{T}})_{k,l} f_k f_l = \sum_{|k|,|l|=1}^N (K^{\gamma,V}_{\mathbb{T}})_{k,l} f_k f_l + \sum_{\max\{|k|,|l|\}\geq N} (K^{\gamma,V}_{\mathbb{T}})_{k,l} f_k f_l.$$
    Since $f\in L^2$ (we can not guarantee this for $N=\infty$ at this point), we can  apply the strengthened Grunsky inequality, which gives 
        $$ \sum_{|k|,|l|=1}^N (K^{\gamma,V}_{\mathbb{T}})_{k,l} f_k f_l \geq - \kappa \sum_{|k|,|l|=1}^N f_k^2.$$
    By Proposition \ref{reg_G}, the second term can be absorbed in $F(n)$ whenever $N=N(n)\to\infty$ and $n/N(n)=O(1)$ as $n\to\infty$ since 
        $$\left|\sum_{\max\{|k|,|l|\}\geq N} (K^{\gamma,V}_{\mathbb{T}})_{k,l} f_k f_l \right|\leq n^2 \sum_{\max\{|k|,|l|\}\geq N} |(K^{\gamma,V}_{\mathbb{T}})_{k,l}| |kl|^{-\frac{1}{2}},$$
    and $z_e\in C^{4,\alpha}$. It thus remains to obtain an appropriate upper bound for
        $$\kappa^{\frac{1}{2}} \sum_{|k|,|l|=1}^N f_k^2,$$
    in terms of the initial $\log$-kernel. Since $\kappa^{\frac{1}{2}}$, we have $\kappa^{\frac{1}{2}} \leq \kappa^{\frac{|k|}{2N}}$ and thus
    $$\kappa^{\frac{1}{2}} \sum_{|k|,|l|=1}^N f_k^2 \leq \sum_{|k|,|l|=1}^N \kappa^{\frac{|k|}{2N}} f_k^2 = \sum_{|k|=1}^{\infty} \kappa^{\frac{|k|}{2N}} f_k^2. $$
    We will divide the right hand side into two terms
    $$\sum_{1\leq u \neq v \leq n} \sum_{|k|=1}^{\infty} \kappa^{\frac{|k|}{2N}} \psi_k(\theta_u)\psi_k(\theta_v) + \sum_{u=1}^n \sum_{|k|=1}^{\infty} \kappa^{\frac{|k|}{2N}} \psi_k(\theta_u)^2.$$
    The first term can be connected to the initial $\log$-kernel. We have
    \begin{align*}
        \sum_{|k|=1}^{\infty} \kappa^{\frac{|k|}{2N}} \psi_k(\theta_u)\psi_k(\theta_v) &=  \sum_{|k|=1}^{\infty} \frac{\kappa^{\frac{|k|}{2N}}}{k} (\cos\theta_u\cos\theta_v + \sin\theta_u\sin\theta_v) \\
        &= \sum_{|k|=1}^{\infty} \frac{\kappa^{\frac{|k|}{2N}}}{k} \cos(\theta_u-\theta_v) \\
        &=- \log|1-\kappa^{\frac{1}{2N}}e^{i(\theta_u-\theta_v)}|.
    \end{align*}
    By Taylor's theorem there exists $a\in(\kappa^{\frac{1}{2N}},1)$ such that
    $$-\log|1-\kappa^{\frac{1}{2N}}e^{i(\theta_u-\theta_v)}| = -\log|1-e^{i(\theta_u-\theta_v)}| - (1-\kappa^{\frac{1}{2N}}) \RP{\frac{e^{i(\theta_u-\theta_v)}}{1-ae^{i(\theta_u-\theta_v)}}}.$$
    Now we use the fact that
        $$-\RP{\frac{e^{i(\theta_u-\theta_v)}}{1-ae^{i(\theta_u-\theta_v)}}} = \frac{a-\cos(\theta_u-\theta_v)}{1+a^2-2a\cos(\theta_u-\theta_v)} \leq \frac{1}{a+1} \leq 1, $$
    to conclude that
        $$\sum_{1\leq u \neq v \leq n} \sum_{|k|=1}^{\infty} \kappa^{\frac{|k|}{2N}} \psi_k(\theta_u)\psi_k(\theta_v) \leq \sum_{1\leq u \neq v \leq n} \widetilde{\mathcal{L}}_0(\theta_u,\theta_v) + n(n-1) (1-\kappa^{\frac{1}{2N}}).$$
    For the second term, we use the estimate
        $$ \sum_{u=1}^n \sum_{|k|=1}^{\infty} \kappa^{\frac{|k|}{2N}} \psi_k(\theta_u)^2  \leq n \sum_{|k|=1}^{\infty} \frac{\kappa^{\frac{|k|}{2N}}}{k} = - n \log(1-\kappa^{\frac{1}{2N}}).$$
    This estimate explains the importance of the prefactor $\kappa^{\frac{1}{2}}\neq 1$ in the beginning. We can now take $N=N(n)$ such that $1-\kappa^{\frac{1}{2N(n)}}\asymp \frac{1}{n}$ as $n\to\infty$. Then $n/N(n)=O(1)$ as $n\to\infty$ and 
        $$\kappa^{\frac{1}{2}} \sum_{|k|=1}^{N(n)} f_k^2 \leq \sum_{1\leq u \neq v \leq n} \mathcal{L}_0(e^{i\theta_u},e^{i\theta_v}) + (n-1) + n\log n,$$
    where the second term may be absorbed in $F(n)$. For the third term, we note that
        $$ \sum_{1\leq u \neq v \leq n} \mathcal{L}_0(\xi^{(n)}_u, \xi^{(n)}_v) = - n\log n, $$
    see, e.g., \cite[\S 3.2]{CourteautJohansson2025}. We can therefore conclude that
        $$ \sum_{1\leq u \neq v \leq n} (\mathcal{G}(e^{i\theta_u},e^{i\theta_v})-G_{0,0}) \geq - \kappa^{\frac{1}{2}} \sum_{1\leq u \neq v \leq n} (\mathcal{L}_0(e^{i\theta_u},e^{i\theta_v})-\mathcal{L}_0(\xi^{(n)}_u, \xi^{(n)}_v)) + F(n), $$
    which finishes the proof.
\end{proof}

We can then use the following estimate for the initial $\log$-kernel.

\begin{prop} \label{L_lb}
For all $\theta\in[0,2\pi]^n$, we have
    $$ \sum_{1\leq u \neq v \leq n} (\mathcal{L}_0(e^{i\theta_u},e^{i\theta_v})-\mathcal{L}_0(\xi^{(n)}_u, \xi^{(n)}_v)) \geq \frac{n}{4} \sum_{u=1}^n t_u^2(\theta). $$
\end{prop}
\begin{proof}
    Recall from \eqref{L0_Taylor} that
    $$\sum_{1\leq u \neq v \leq n} (\mathcal{L}_0(e^{i\theta_u},e^{i\theta_v})-\mathcal{L}_0(\xi^{(n)}_u, \xi^{(n)}_v)) = \int_0^1 \p_n''(a) (1-a) da,$$
    with
    \begin{align*}
        \p_n''(a) &= \sum_{1\leq u \neq v \leq n} \frac{(t_u-t_v)^2}{4 \sin^2(\frac{\beta_u-\beta_v}{2})},
    \end{align*}
    where $\beta_u(a) = \alpha^{(n)}_u + a t_u(\theta)$ and $t_u(\theta)$ is as in \eqref{t_u_def}. We can then use the estimate
    $$\p_n''(a) \geq \frac{1}{4}\sum_{1\leq u \neq v \leq n} (t_u-t_v)^2 = \frac{1}{4}\sum_{u,v=1}^n (t_u-t_v)^2 = \frac{n}{2} \sum_{u=1}^n t_u^2, $$  
    to conclude that
        $$ \sum_{1\leq u \neq v \leq n} (\mathcal{L}_0(e^{i\theta_u},e^{i\theta_v})-\mathcal{L}_0(\xi^{(n)}_u, \xi^{(n)}_v)) \geq \frac{n}{4} \sum_{u=1}^n t_u^2. $$
    We therefore obtain the desired result.
\end{proof}

As a consequence, we obtain the following information.

\begin{prop} \label{t_u^2}
    Suppose that $z_e\in C^{4,\alpha}(\mathbb{T})$ and $V\in C^{3,\alpha}(\gamma)$. Then for large enough $K\in\R$, there exists $C_4>0$ such that for all $n\in\N$,
        $$\sup_{s\in [0,1]} \sup_{\theta\in E_{n,K}^{[s]}} \sum_{u=1}^n t_u^2(\theta) \leq C_4 . $$
\end{prop}
\begin{proof}
    By definition of $E_{n,K}^{[s]}$, we have
    $$\sum_{1\leq u\neq v \leq n} (\mathcal{L}_s(e^{i\theta_u},e^{i\theta_v})-\mathcal{L}_0(\xi^{(n)}_u, \xi^{(n)}_v)-G_{0,0}) \leq Kn.$$
    It then follows from Proposition \ref{G_lb} that
    $$(1-s\kappa^{\frac{1}{2}}) \sum_{1\leq u \neq v \leq n} (\mathcal{L}_0(e^{i\theta_u},e^{i\theta_v})-\mathcal{L}_0(\xi^{(n)}_u, \xi^{(n)}_v)) + s F(n) \leq Kn, $$
    Therefore, by Proposition \ref{L_lb} and the estimates $1-s\kappa^{\frac{1}{2}}\geq 1-\kappa^{\frac{1}{2}} >0$,
    $$ \frac{n}{4} \sum_{u=1}^n t_u^2(\theta) \leq \frac{Kn-sF(n)}{1-\kappa^{\frac{1}{2}}}.$$
    Since $F(n)=O(n)$ as $n\to\infty$ the desired result follows.
\end{proof}

It leads to the conclusion below, which shows that \eqref{EDI_WC} holds. 

\begin{prop} \label{T_w_conv_rate}
    Suppose that $z_e\in C^{4,\alpha}(\mathbb{T})$ and $V\in C^{3,\alpha}(\gamma)$. Then for large enough $K\in\R$, there exists $C_5>0$ such that for all $f\in C^1(\mathbb{T})$ and $n\in\N$,
        $$ \sup_{s\in [0,1]} \sup_{z^{(n)}\in D_{n,K}^{[s]}} \left| \int f \, d\lambda_{z^{(n)}} \right| \leq C_5 \| f^{(1)} \|_\infty \sqrt{n}.$$
\end{prop}
\begin{proof}
    Observe that
    $$\left| \int f \, d\lambda_{z^{(n)}} \right| \leq \sum_{u=1}^n |f(z_u)-f(\xi_u^{(n)})| + \left| \int f \, d\lambda_{\xi^{(n)}} \right|, $$
    where, by Proposition \ref{CR_FEK}, we have the estimate
    $$\left| \int f \, d\lambda_{\xi^{(n)}} \right| \leq C_1 \| f^{(1)} \|_\infty .$$
    On the other hand, by the mean value theorem and Hölder's inequality,
    $$\sum_{u=1}^n |f(z_{(u)})-f(\xi_u^{(n)})| \leq \| f^{(1)} \|_\infty \sum_{u=1}^n |\theta_{(u)}-\alpha_u^{(n)}| \leq \| f^{(1)} \|_\infty \sqrt{n} \left(\sum_{u=1}^n t_u^2(\theta)\right)^{\frac{1}{2}},$$
    where $z^{(n)}=(e^{i\theta_1},\dots,e^{i\theta_n})$ and $z_{(u)}=e^{i\theta_{(u)}}$. It then remains to use Proposition \ref{t_u^2}.
\end{proof}

At this point, we can show the following, which proves Proposition \ref{LS_lim_s_proof}~\ref{est_V}.

\begin{prop} \label{LS_const}
    Suppose that $z_e\in C^{7,\alpha}(\mathbb{T})$ and that $V\in C^{6,\alpha}(\gamma)$. Let $g\in C^{4,\alpha}(\mathbb{T})$ with $\int g d\mu_{\mathbb{T}}=0$ and let
    $h_s:\mathbb{T}\to\R$ be the solution of the master equation \eqref{ME} given $g$. Then for large enough $K\in\R$, there exists $C(\|g\|_{4,\alpha})>0$ such that for all $n\in\N$,
        $$  \sup_{s\in [0,1]} \sup_{z^{(n)}\in E_{n,K}^{[s]}} \left| \frac{1}{n} \iint [h_s(u) \partial_1 +  h_s(v)\partial_2] \mathcal{L}_s(u,v) d\lambda_{z^{(n)}}(u)d\lambda_{z^{(n)}}(v) \right| \leq C(\|g\|_{4,\alpha}).$$
\end{prop}
\begin{proof}
Recall that $\mathcal{L}_s(u,v)=\mathcal{L}_0(u,v) + s \mathcal{G}(u,v)$ and denote
\begin{align*}
    \mathcal{F}_1(u,v) = [h_s(u) \partial_1 +  h_s(v)\partial_2] \mathcal{L}_0(u,v), \quad
    \mathcal{F}_2(u,v) = [h_s(u) \partial_1 +  h_s(v)\partial_2] \mathcal{G}(u,v).
\end{align*}
We will first focus on $\mathcal{F}_1$. We have,
$$\mathcal{F}_1(u,v) = \IP{\frac{uh_s(u)-vh_s(v)}{u-v}},$$
After plugging in the Fourier series $h_s(u) = \sum_{|k|\geq 1} \hat{h}_{s,k} u^k $, this becomes
\begin{align*}
    \mathcal{F}_1(u,v) = 2\, \IP{\sum_{k\geq 1} \hat{h}_{s,k} \frac{u^{k+1}-v^{k+1}}{u-v}} =2\, \IP{\sum_{k\geq 1} \hat{h}_{s,k} \sum_{l=0}^{|k|} u^{\operatorname{sgn}(k)(k-l)} v^{\operatorname{sgn}(k)l}}.
\end{align*}
Hence,
\begin{align} \label{F_1_ini_est}
        &\left|\iint \mathcal{F}_1(u,v) d\lambda_{z^{(n)}}(u)d\lambda_{z^{(n)}}(v)\right| \notag \\
        &\qquad\leq 2\sum_{|k|\geq 1} |\hat{h}_{s,k}| \sum_{l=0}^{|k|} \left|\int u^{\operatorname{sgn}(k)(k-l)} d\lambda_{z^{(n)}}(u)\right| \left|\int v^{\operatorname{sgn}(k)l} d\lambda_{z^{(n)}}(v)|\right|,
\end{align}
and thus by Proposition \ref{T_w_conv_rate}, there exists $C_1>0$ such that
$$ \sup_{s\in [0,1]} \sup_{z^{(n)}\in E_{n,K}^{[s]}} \left|\iint \mathcal{F}_1(u,v) d\lambda_{z^{(n)}}(u)d\lambda_{z^{(n)}}(v)\right| \leq C_1 n \sum_{|k|\geq 1} |\hat{h}_{s,k}| \sum_{l=0}^{|k|} (|k|-l) l. $$
This allows us to conclude that
$$\sup_{s\in [0,1]} \sup_{z^{(n)}\in E_{n,K}^{[s]}} \left|\iint \mathcal{F}_1(u,v) d\lambda_{z^{(n)}}(u)d\lambda_{z^{(n)}}(v)\right| \leq C_1 n \sum_{k\geq 1} |\hat{h}_{s,k}| |k|^3 .$$
By Proposition \ref{reg_h_s}, the latter remains bounded whenever $g \in C^{4,\alpha}$, $z_e\in C^{7,\alpha}$ and $V\in C^{6,\alpha}$.

We will now consider $\mathcal{F}_2$. By symmetry, we have
$$\iint \mathcal{F}_2(u,v) d\lambda_{z^{(n)}}(u)d\lambda_{z^{(n)}}(v) = 2 \iint h_s(u) \mathcal{G}^{(1,0)}(u,v) d\lambda_{z^{(n)}}(u)d\lambda_{z^{(n)}}(v). $$
Hence,
\begin{equation} \label{F_2_ini_est}
    \left|\iint \mathcal{F}_2(u,v) d\lambda_{z^{(n)}}(u)d\lambda_{z^{(n)}}(v) \right| \leq 2 \sum_{|k|,|l|\geq 1} |G_{k,l}| \left|\int h_s \phi_k^{(1)} d\lambda_{z^{(n)}} \right| \left|\int \phi_l d\lambda_{z^{(n)}} \right|,
\end{equation}
and an application of Proposition \ref{T_w_conv_rate} gives $C_2>0$ such that
$$\sup_{s\in [0,1]} \sup_{z^{(n)}\in E_{n,K}^{[s]}} \left|\iint \mathcal{F}_2(u,v) d\lambda_{z^{(n)}}(u)d\lambda_{z^{(n)}}(v) \right| \leq C_2 n \sum_{|k|,|l|\geq 1} |G_{k,l}| \|(h_s \phi_k^{(1)})^{(1)}\|_{\infty} \|\phi_l^{(1)}\|_{\infty}. $$
Since $h_s \in C^1$, there then exists $c_1>0$ such that
$$\sup_{s\in [0,1]} \sup_{z^{(n)}\in E_{n,K}^{[s]}} \left|\iint \mathcal{F}_2(u,v) d\lambda_{z^{(n)}}(u)d\lambda_{z^{(n)}}(v) \right| \leq C_2 c_1 n \sum_{|k|,|l|\geq 1} |G_{k,l}| k^2 |l|. $$
By Proposition \ref{reg_G}, the latter remains bounded whenever $z_e\in C^{6,\alpha}$ and $V\in C^{5,\alpha}$.
\end{proof}

\begin{rem}
    The above is the bottleneck for the condition $z_e \in C^{7,\alpha}(\mathbb{T})$ in Theorem \ref{av_s}.
\end{rem}

In order to prove that Proposition~\ref{lim_E_V} is satisfied, we require the following intermediate result.

\begin{prop} \label{s_LS_const}
    Suppose that $z_e\in C^{7,\alpha}(\mathbb{T})$ and $V\in C^{6,\alpha}(\gamma)$. Let $g\in C^{4,\alpha}(\mathbb{T})$ with $\int g d\mu_{\mathbb{T}}=0$. Then there exists $C(\|g\|_{4,\alpha})>0$ such that for all $n\in\N$,
        $$ \sup_{s\in [0,1]} \mathbb{E}_{n,\beta}^{[s]}\left[e^{\sum_{k=1}^n g(z_k)}\right] \leq C(\|g\|_{4,\alpha}).$$
\end{prop}
\begin{proof}
    We will argue similarly as in the proof of Proposition \ref{LS_lim_s_proof}. It follows from \eqref{est_V_glob_uni} and Proposition \ref{reg_h_s} that
        $$\mathbb{E}_{n,\beta}^{[s]}\left[e^{\sum_{k=1}^n g(z_k)}\right] \leq \widetilde{\mathbb{E}}_{n,\beta}^{[s]}\left[e^{\widetilde{\mathcal{V}}_{g,h,\frac{1}{n}}^{[s]}(t)} \mathds{1}_{E_{n,K}^{[s]}}(t) \right] + e^{c_3(\|g\|_{2,\alpha}) n} \widetilde{\mathbb{E}}_{n,\beta}^{[s]}\left[\mathds{1}_{(E_{n,K}^{[s]})^c}(t)\right].$$
    By Proposition \ref{est_D_prop}, we have
        $$ \sup_{s\in[0,1]} \widetilde{\mathbb{E}}_{n,\beta}^{[s]}\left[\mathds{1}_{(E_{n,K}^{[s]})^c}(t)\right] \leq e^{C_3n-\frac{\beta}{2}Kn}.$$
    Hence for large enough $K$, the second term is $o(1)$ as $n\to\infty$, uniformly in $s\in[0,1]$ and $\|g\|_{2,\alpha}$. We can now use \eqref{V_form}, which says that 
        $$ \mathcal{V}_{g,h_s,\frac{1}{n}}(z^{(n)}) = \mathcal{V}_{g,h_s} - \frac{\beta}{2n} \iint [h_s(u) \partial_u +  h_s(v)\partial_v] \mathcal{L}_s(u,v) d\lambda_{z^{(n)}}(u)d\lambda_{z^{(n)}}(v) + o(1), $$
    where $o(1)$ is uniform in $s\in[0,1]$ and $\|g\|_{2,\alpha}$. Indeed, this holds due to  Proposition~\ref{Taylor_err} and the assumption \eqref{EDI_WC}, which is Proposition \ref{T_w_conv_rate}. The desired result then follows from Proposition~\ref{LS_const}.
\end{proof}

In its turn this allows us to show that \eqref{err_0} holds. As a consequence, Proposition~\ref{lim_E_V} is proven.

\begin{prop}
    Suppose that $z_e\in C^{7,\alpha}(\mathbb{T})$ and that
    $V\in C^{6,\alpha}(\gamma)$. Let $g\in C^{4,\alpha}(\mathbb{T})$ with
    $\int g d\mu_{\mathbb{T}}=0$ and let
    $h_s:\mathbb{T}\to\R$ be the solution of the master equation
    \eqref{ME} given $g$. Then for large enough $K\in\R$, we have
    \begin{equation} \label{T_err_0}
    \lim_{n\to\infty}\sup_{s\in[0,1]}
    \mathbb{E}_{n,\beta}^{[s]}\left[
        \frac{1}{n}
        \left|\iint
        [h_s(u)\partial_u+h_s(v)\partial_v]\mathcal{L}_s(u,v)
        d\lambda_{z^{(n)}}(u)d\lambda_{z^{(n)}}(v)\right|
        \mathds{1}_{D_{n,K}^{[s]}}(z^{(n)})
    \right]=0.
    \end{equation}
\end{prop}
\begin{proof}
We will use Proposition \ref{s_LS_const} to refine the estimates
from Proposition \ref{LS_const}. If $f\in C^{4,\alpha}(\mathbb{T})$ and
$\int f d\mu_{\mathbb{T}}=0$, Proposition \ref{s_LS_const},
applied to $\pm f$, and the inequality
$|x|\leq e^x+e^{-x}$ yield
    $$ \sup_{n\in\N}\sup_{s\in[0,1]}
    \mathbb{E}_{n,\beta}^{[s]}
        \left|\int f d\lambda_{z^{(n)}}\right|
    <\infty. $$
Applying this to the real and imaginary parts of $u\mapsto u^{\pm j}$
and to $\phi_{\pm}$, shows that for every $j\in\Z_{\geq 1}$,
there exists $C_j>0$ such that
\begin{equation} \label{c_j_est}
    \sup_{n\in\N}\sup_{s\in[0,1]}
    \max\left\{
        \mathbb{E}_{n,\beta}^{[s]}\left|\int u^{\pm j} d\lambda_{z^{(n)}}(u)\right|,
        \mathbb{E}_{n,\beta}^{[s]}\left|\int \phi_{\pm j} d\lambda_{z^{(n)}}\right|   \right\}\leq C_j.
\end{equation}
On the other hand, by Proposition \ref{T_w_conv_rate}, there exists $c>0$ such that for all $n\in\N$,
\begin{equation} \label{cj_sqrtn_est}
    \sup_{s\in [0,1]} \sup_{z^{(n)}\in D_{n,K}^{[s]}} \max\left\{\left|\int u^{\pm j} d\lambda_{z^{(n)}}(u)\right|, \left|\int \phi_{\pm j} d\lambda_{z^{(n)}}\right|\right\}\leq C j \sqrt{n}.
\end{equation}
We will first consider $\mathcal{F}_1$. Following the proof of Proposition \ref{LS_const}, we have
\begin{align*}
    \mathcal{F}_1(u,v) = 2\, \IP{\sum_{k\geq 1} \hat{h}_{s,k} \sum_{j=0}^{|k|} u^{\operatorname{sgn}(k)(k-j)} v^{\operatorname{sgn}(k)j}}.
\end{align*}
Let $N\in\Z_{\geq 1}$. For $k\in\{1,\dots,N\}$, we will apply \eqref{c_j_est} to $u^{\operatorname{sgn}(k)(k-j)}$ and \eqref{cj_sqrtn_est} to $v^{\operatorname{sgn}(k)j}$. For $k>N$, we will apply \eqref{cj_sqrtn_est} to both factors. Doing so, we obtain
\begin{align*}
    & \mathbb{E}_{n,\beta}^{[s]}\left[
        \left|\iint
        \mathcal{F}_1(u,v)
        d\lambda_{z^{(n)}}(u)d\lambda_{z^{(n)}}(v)\right|
        \mathds{1}_{D_{n,K}^{[s]}}(z^{(n)})
    \right] \\
    &\qquad\leq
        2C\sqrt{n}\sum_{k=1}^{N} |\hat{h}_{s,k}|
            \sum_{j=1}^{k-1}j C_{k-j}
        +2C^2n\sum_{k>N}|\hat{h}_{s,k}|\sum_{j=1}^{k-1}j(k-j).
\end{align*}
By Proposition \ref{reg_h_s}, there exists $c_1>0$ such that
    $$ \sup_{s\in[0,1]}|\hat{h}_{s,k}|
        \leq \frac{c_1}{k^{4+\alpha}},\quad k\geq 1. $$
Therefore,
\begin{align*}
    & \sup_{s\in[0,1]} \mathbb{E}_{n,\beta}^{[s]}\left[
        \left|\iint
        \mathcal{F}_1(u,v)
        d\lambda_{z^{(n)}}(u)d\lambda_{z^{(n)}}(v)\right|
        \mathds{1}_{D_{n,K}^{[s]}}(z^{(n)})
    \right] \\
    &\qquad\leq
        2Cc_1\sqrt{n}\sum_{k=1}^{N}
            \sum_{j=1}^{k-1}j C_{k-j}
        +2C^2c_1n\sum_{k>N}\frac{1}{k^{1+\alpha}}.
\end{align*}
Dividing by $n$, the first term tends to $0$ as $n\to\infty$. We can then let $N\to\infty$ to conclude that
    $$ \lim_{n\to\infty}\sup_{s\in[0,1]}
    \mathbb{E}_{n,\beta}^{[s]}\left[ \frac{1}{n}
        \left|\iint
        \mathcal{F}_1(u,v)
        d\lambda_{z^{(n)}}(u)d\lambda_{z^{(n)}}(v)\right|
        \mathds{1}_{D_{n,K}^{[s]}}(z^{(n)})
    \right]=0. $$

We will now consider $\mathcal{F}_2$. Following the proof of Proposition \ref{LS_const}, we have
$$ \iint  \mathcal{F}_2(u,v) d\lambda_{z^{(n)}}(u)d\lambda_{z^{(n)}}(v) = 2 \sum_{|k|,|l|\geq 1} G_{k,l} \left(\int h_s \phi_k^{(1)} d\lambda_{z^{(n)}} \right) \left(\int \phi_l d\lambda_{z^{(n)}} \right). $$
By Proposition \ref{T_w_conv_rate}, there exists $c_2,c_3>0$ such that
    $$ \sup_{s\in [0,1]} \sup_{z^{(n)}\in D_{n,K}^{[s]}} \left|\int h_s\phi_k^{(1)}d\lambda_{z^{(n)}}\right|
        \leq c_2\sqrt{n} \sup_{s\in [0,1]} \|(h_s\phi_k^{(1)})^{(1)}\|_\infty
        \leq c_3 k^2\sqrt{n}. $$
For $|l|\,\in \{1,\dots,N\}$, we will apply \eqref{c_j_est} to $\int \phi_l d\lambda_{z^{(n)}}$, while for $|l|\,>N$, we will apply \eqref{cj_sqrtn_est}. Doing so, we obtain
\begin{align*}
    & \sup_{s\in[0,1]}\mathbb{E}_{n,\beta}^{[s]}\left[
        \left|\iint
        \mathcal{F}_2(u,v)
        d\lambda_{z^{(n)}}(u)d\lambda_{z^{(n)}}(v)\right|
        \mathds{1}_{D_{n,K}^{[s]}}(z^{(n)})
    \right] \\
    &\qquad\leq
        2c_3 \sqrt{n}
        \sum_{|k|\geq 1}\sum_{1\leq |l|\leq N}
            |G_{k,l}|k^2 C_{|l|}
        +2C c_3 n\sum_{\substack{|k|\geq 1,\\|l|>N}}
            |G_{k,l}|k^2|l|.
\end{align*}
Dividing by $n$, the first term tends to $0$ as $n\to\infty$. We can then let $N\to\infty$ to conclude that
    $$ \lim_{n\to\infty}\sup_{s\in[0,1]}
    \mathbb{E}_{n,\beta}^{[s]}\left[\frac{1}{n}
        \left|\iint
        \mathcal{F}_2(u,v)
        d\lambda_{z^{(n)}}(u)d\lambda_{z^{(n)}}(v)\right|
        \mathds{1}_{D_{n,K}^{[s]}}(z^{(n)})
    \right] = 0. $$
Combining both limits then proves \eqref{T_err_0}.
\end{proof}

Since Proposition \ref{lim_E_V}--\ref{est_V_glob} are satisfied, by Proposition \ref{LS_lim_s_proof}, Theorem \ref{av_s} is proven as well.

\subsection{Other representations of the limit} \label{S_RV}

In what follows, we will derive different representations for the right hand side of \eqref{int_ls}, independent of $h_s$. Crucial will be to have the differentiation rule \eqref{T_D} and integration by parts \eqref{IBP}, which hold for $\mathbb{T}$. We may derive similar results on contours $\tau$ with a more general differentiation rule \eqref{tau_D}, however then there will also be some contributions from the endpoints.
\medbreak

The first representation involves the real-valued $L^2$-basis.

\begin{prop} \label{V_coeff} 
    Let $f=f_0+\vec{f}^t\psi\in C^{1,\alpha}(\mathbb{T})$ and $g=\vec{g}^t\psi\in L^2_0(\mathbb{T},\mu_{\mathbb{T}})$. If $h_s:\mathbb{T}\to\R$ is the solution of the master equation \eqref{ME} given $g$, then
    $$ \int h_s f^{(1)} d\mu_{\mathbb{T}} = \frac{1}{\beta} \vec{f} L_s^{-1} \vec{g}^t. $$
\end{prop}
\begin{proof}
Write $h_s = \vec{h}_s \psi^t$ and note that $f^{(1)} = \vec{f} D\psi^t$ so that
\begin{align*}
    &\int h_s g^{(1)} d\mu_{\mathbb{T}} = \int \vec{f} D\psi^t \psi \vec{h}_s^t d\mu_{\mathbb{T}} = \vec{f} D\Lambda \vec{h}_s^t,
\end{align*}
It then remains to use that $D\Lambda = \Lambda^{\frac{1}{2}} D\Lambda^{\frac{1}{2}}$ and
$$(\Lambda^{\frac{1}{2}} D\Lambda^{\frac{1}{2}})\vec{h}_s^t = \frac{1}{\beta} L_s^{-1} \vec{g}^t,$$
see \eqref{ME_FOUR}, to finish the proof.
\end{proof}

\begin{rem}
    Note that in order to obtain this representation, we only need to solve the master equation \eqref{ME_FOUR} for $D\vec{h}_s$, however for the change of variables argument in Section~\ref{proofs_strat} to go through we need to be able to solve for $\vec{h}_s$.
\end{rem}

The second representation involves the Dirichlet inner product on $\C\setminus\gamma$, see \eqref{Dir_IP}.

\begin{prop} \label{V_DE}
    Let $f,g\in C^{1,\alpha}(\gamma)$ with $f_0,g_0=0$ and let $h_1:\mathbb{T}\to\R$ be the solution of the master equation for $s=1$ given $gz_e$. Suppose that $z_e\in C^{4,\alpha}(\mathbb{T})$. Then,
        $$ \int h_1 \cdot (fz_e)^{(1)} d\mu_{\mathbb{T}} = \frac{2}{\beta} \mathcal{D}_{\C\setminus\gamma}(f,g).  $$
    In particular, 
    $$ \int h_1 \cdot (gz_e)^{(1)} d\mu_{\mathbb{T}} = \frac{2}{\beta} \mathcal{D}_{\C\setminus\gamma}(g).  $$
\end{prop}
\begin{proof}
It follows from Proposition \ref{reg_h_s} that $h_1\in C^{1,\alpha}(\gamma)$. The integration by parts rule \eqref{IBP} allows us to write
    $$ \int h_1 \cdot (fz_e)^{(1)} d\mu_{\mathbb{T}} = - \int h_1^{(1)} \cdot (fz_e) d\mu_{\mathbb{T}},$$
or after using the fact that $\mu_{\gamma,V} = z_e\#\mu_{\mathbb{T}}$,
$$ \int h_1 \cdot (fz_e)^{(1)} d\mu_{\mathbb{T}} = \int f d\tilde{\mu}_{\gamma,V}(v),$$
where the (signed) measure $\tilde{\mu}_{\gamma,V}$ on $\gamma$ is defined through 
 $$ d\tilde{\mu}_{\gamma,V}(v) = -h_1^{(1)}(z_e^{-1}(v)) d\mu_{\gamma,V}(v). $$
We will now derive an alternative expression for $d\tilde{\mu}_{\gamma,V}$ from the master equation. After applying the integration by parts rule, the latter can be written as
     $$  \frac{1}{\beta}(gz_e)(u)= - {\rm p.v.} \int (h_1z_e)^{(1)}(v) \mathcal{L}_1(u,v) d\mu_{\mathbb{T}}(v),\quad u\in\mathbb{T}. $$
More precisely, the right hand side is given by
$$ - \int (h_1z_e)^{(1)}(v) \frac{1}{|z_e(u)-z_e(v)|} d\mu_{\mathbb{T}}(v) - \frac{1}{2} \int (h_1z_e)^{(1)} \cdot (Vz_e) d\mu_{\mathbb{T}}. $$
Since $\mu_{\gamma,V} = z_e\#\mu_{\mathbb{T}}$, the first term can be written as
$$- \int (h_1z_e)^{(1)}(z_e^{-1}(v)) \log \frac{1}{|z_e(u)-v|} d\mu_{\gamma,V}(v).$$
We can then reparametrize the master equation as
    $$  \frac{1}{\beta}g(u)= \int \log \frac{1}{|u-v|} d\tilde{\mu}_{\gamma,V}(v) - \frac{1}{2} \int h_1^{(1)} \cdot (Vz_e) d\mu_{\mathbb{T}} ,\quad u\in\gamma, $$
In that case, we have
    $$  \frac{1}{\beta}g(u)= U^{\tilde{\mu}_{\gamma,V}}(u) - \frac{1}{2} \int h_1^{(1)} (Vz_e) d\mu_{\mathbb{T}},\quad u\in\gamma. $$
As such, $g$ admits an extension that is bounded and harmonic on the connected components of $\C\setminus\gamma$, which we will denote $g_{\rm ext}$. Since $g\in C^{1}$, the same holds for $U^{\tilde{\mu}_{\gamma,V}}$, hence by \cite[Thm. II.1.5]{SaffTotik1997}, we have
    $$ d\tilde{\mu}_{\gamma,V}(z) =  - \frac{1}{2\pi} ((\partial_{\mathfrak{n}_+} U^{\tilde{\mu}_{\gamma,V}})(z) + (\partial_{\mathfrak{n}_-} U^{\tilde{\mu}_{\gamma,V}})(z)) |dz|. $$
As a consequence,
    $$ d\tilde{\mu}_{\gamma,V}(z) = -\frac{1}{2\pi\beta} (\partial_{\mathfrak{n}_+} g_{\rm ext} + \partial_{\mathfrak{n}_-} g_{\rm ext}) |dz|, $$
and thus
$$ \int (h_1z_e) \cdot (fz_e)^{(1)} d\mu_{\mathbb{T}} = -\frac{1}{2\pi\beta} \int_\gamma f_{\rm ext}\cdot (\partial_{\mathfrak{n}_+}g_{\rm ext} + \partial_{\mathfrak{n}_-}g_{\rm ext}) |dz|.$$
It then remains to apply Green's theorem.
\end{proof}

Finally, we require the following result.

\begin{prop} \label{repr_eq_dens}
    Suppose that $z_e\in C^{1,\alpha}(\mathbb{T})$. Then,
        $$ \log |z_e^{(1)}\circ z_e^{-1}| \, = \log \frac{1}{2\pi} - \log w_{\gamma,V}. $$
\end{prop}
\begin{proof}
Since $z_e\circ z_e^{-1} = {\rm Id}$, we have $(z_e^{(1)}\circ z_e^{-1})\cdot
(z_e^{-1})^{(1)} = 1$. By definition, see \eqref{eq_par_def}, we have $z_e^{-1} = (\tau F_{\tau}^{-1} F_{\gamma,V}\gamma^{-1})$ for $\tau=\mathbb{T}$. For $\tau=\mathbb{T}$, $\tau(t)=e^{it}$ and $F_{\tau}(t)=t/(2\pi)$ and thus $(z_e^{-1}\gamma) = e^{2\pi i F_{\gamma,V}}$. Consequently, 
        $$ (z_e^{-1})^{(1)}(\gamma(t)) = (z_e^{-1}\gamma)'(t) = e^{2\pi i F_{\gamma,V}(t)} 2\pi i F_{\gamma,V}'(t) = e^{2\pi i F_{\gamma,V}(t)} 2\pi i (w_{\gamma,V}\gamma)(t), $$
which then leads to the stated result.
\end{proof}

\subsection{Proof of Theorem \ref{PF} \& \ref{ThmSz}} \label{proofs_mr}

\begin{proof}[Proof of Theorem \ref{ThmSz}]
First replace $g$ by $g-\int g d\mu_{\gamma,V}$. This centers the linear statistic and leaves the Dirichlet pairings unchanged. We can then apply Theorem \ref{av_s} to $gz_e$ with $s=1$. By Proposition~\ref{V_DE}, we then have
\begin{align*}
    \lim_{n\to\infty} \log \mathbb{E}_{n,\beta}^{\gamma,V}\left[e^{\sum_{k=1}^n g(z_k)}\right] = & (\frac{2}{\beta}-1) \mathcal{D}_{\C\setminus\gamma}(g,\log|z_e^{(1)}\circ z_e^{-1}|) + \frac{1}{\beta} \mathcal{D}_{\C\setminus\gamma}(g).
\end{align*}
It remains to use Proposition \ref{repr_eq_dens} to conclude the desired result.
\end{proof}

For technical reasons, in order to prove Theorem \ref{PF}, we require the following result.
\begin{prop} \label{s_CLT_MV}
    Suppose that $z_e\in C^{7,\alpha}(\mathbb{T})$. For all $g\in C^{4,\alpha}(\mathbb{T})$ with $\int g d\mu_{\mathbb{T}}=0$, there exists $C(\|g\|_{4,\alpha})>0$ such that for all $s\in[0,1]$ and $n\in\N$,
        $$ \mathbb{E}_{n,\beta}^{[s]}\Big[ \sum_{k=1}^n g(z_k)\Big] \leq C(\|g\|_{4,\alpha}),\quad \mathbb{E}_{n,\beta}^{[s]}\Big[\Big(\sum_{k=1}^n g(z_k)\Big)^2\Big] \leq C(\|g\|_{4,\alpha}).$$
\end{prop}
\begin{proof}
    Consider
    $$ \p_{n,\beta}^{[s]}(z) = \mathbb{E}_{n,\beta}^{[s]}[e^{z \sum_{k=1}^{n} g(z_k)}],\quad |z|\leq 1. $$
    Note that
        $$ |\p_{n,\beta}^{[s]}(z)| \leq \mathbb{E}_{n,\beta}^{[s]}[e^{\RP{z} \sum_{k=1}^{n} g(z_k)}] \leq \mathbb{E}_{n,\beta}^{[s]}[e^{\pm \sum_{k=1}^{n} g(z_k)}] $$
    Hence, by Proposition \ref{s_LS_const}, there exists $C_7(\|g\|_{4,\alpha})>0$ such that for all $|z|\leq 1$,
        $$ |\p_{n,\beta}^{[s]}(z)| \leq C_7(\|g\|_{4,\alpha}). $$
    Since 
        $$ |(\p_{n,\beta}^{[s]})'(0)| = \mathbb{E}_{n,\beta}^{[s]}\Big[ \sum_{k=1}^n g(z_k)\Big],\quad |(\p_{n,\beta}^{[s]})''(0)| = \mathbb{E}_{n,\beta}^{[s]}\Big[\Big(\sum_{k=1}^n g(z_k)\Big)^2\Big] , $$
    the desired result then follows from Cauchy’s integral formula.        
\end{proof}

We are now ready to show Theorem \ref{PF}.

\begin{proof}[Proof of Theorem \ref{PF}] 
Our proof is based on the fact that
    $$ \log \frac{Z_{n,\beta}^{\gamma,V}}{Z_{n,\beta}^{\mathbb{T},0}} = \int_0^1 \frac{\partial_s Z_{n,\beta}^{[s]}}{Z_{n,\beta}^{[s]}} ds, $$
similarly as in \cite[\S 5.3]{CourteautJohanssonViklund2026}. It follows from \eqref{s_Z} and \eqref{H_n^s_alt} that
\begin{align*}
\frac{\partial_s Z_{n,\beta}^{[s]}}{Z_{n,\beta}^{[s]}} 
&= -\frac{\beta}{2} \mathbb{E}_{n,\beta}^{[s]} \Big[ \sum_{\substack{k,l=1 \\ k \neq l}}^n \mathcal{G}(z_k, z_l) + \sum_{k=1}^n (Vz_e)(z_k) \Big] + \mathbb{E}_{n,\beta}^{[s]} \Big[ \sum_{k=1}^n \log |z_e^{(1)}(z_k)| \Big] \\
&= -\frac{\beta}{2} \mathbb{E}_{n,\beta}^{[s]} \Big[ \sum_{k,l=1}^n \mathcal{G}(z_k, z_l) \Big] + \mathbb{E}_{n,\beta}^{[s]} \Big[ \sum_{k=1}^n m(z_k) \Big],
\end{align*}
in terms of $ m = (1-\frac{\beta}{2})\log|z_e^{(1)}|$. Denote $X_k(z) = \sum_{j=1}^n \psi_k(z_j)$ and expand $m = m_0 + \sum_{|k| \ge 1} m_k \psi_k$, then
\begin{equation*}
\frac{\partial_s Z_{n,\beta}^{[s]}}{Z_{n,\beta}^{[s]}} = -\frac{\beta}{2} \mathbb{E}_{n,\beta}^{[s]} \Big[ \sum_{|k|,|l| \geq 1} K_{k,l} X_k X_l \Big] + \mathbb{E}_{n,\beta}^{[s]} \Big[ \sum_{|k| \geq 1} m_k X_k \Big] + M_n,
\end{equation*}
where $M_n = -\frac{\beta}{2}n^2 G_{0,0} + n m_0 $. At this point, we would like to interchange the expectations and series, pass the limit through the initial integral and then through the series. In order to justify this, we will use Fubini's theorem, the dominated convergence theorem and appropriate bounds on $K_{k,l}$, $m_k$ and $\mathbb{E}_{n,\beta}^{[s]} [X_kX_l]$ and $\mathbb{E}_{n,\beta}^{[s]} [X_k]$. The starting point is to use Proposition \ref{s_CLT_MV}. In order to apply it, we need normalize $X_k$ such that $\|\cdot\|_{4,a}=1$, say with $a=\alpha/4$. Clearly, $\|\psi_k\|_{4,a} \leq c_1 |k|^{4+a-\frac{1}{2}}$ and therefore
\begin{align*}
    \mathbb{E}_{n,\beta}^{[s]} [|X_k X_l|] &\leq \mathbb{E}_{n,\beta}^{[s]} [X_k^2]^{\frac{1}{2}} \mathbb{E}_{n,\beta}^{[s]} [X_l^2]^{\frac{1}{2}} \leq c_1^2 C(1) |kl|^{4+a-\frac{1}{2}}, \\
    \mathbb{E}_{n,\beta}^{[s]} [X_k] &\leq c_1 C(1) |k|^{4+a-\frac{1}{2}}.
\end{align*}
Whenever $z_e\in C^{11,\alpha}$ and $m\in C^{5,\alpha}$, we can then use Proposition \ref{reg_G} to obtain the estimates
    $$ |K_{k,l}| \leq 2 |kl|^{\frac{1}{2}} \frac{C_{5+\frac{\alpha}{2},5+\frac{\alpha}{2}}}{|kl|^{5+\frac{\alpha}{2}}},\quad |m_k| \leq |k|^{\frac{1}{2}} \frac{1}{|k|^{5+\alpha}}, $$
Denote $g_{a,b} = a X_k + b X_l$ and note that
\begin{align*}
\mathbb{E}_{n,\beta}^{[s]} [ X_k X_l ] = \left. \partial_a \partial_b \mathbb{E}_{n,\beta}^{[s]} [e^{g_{a,b}}]\right|_{a,b=0},\quad \mathbb{E}_{n,\beta}^{[s]} [ X_k ] = \left. \partial_a \mathbb{E}_{n,\beta}^{[s]} [e^{g_{a,0}}] \right|_{a=0}.
\end{align*}

By Theorem \ref{av_s} and Proposition \ref{V_coeff}, we obtain
\begin{align*}
\lim_{n \to \infty} \mathbb{E}_{n,\beta}^{[s]} [ X_k X_l ] 
& = \left. \partial_a \partial_b \exp \Big( \frac{1}{2\beta} \vec{g}_{a,b}^t L_s^{-1} \vec{g}_{a,b} +  \frac{s}{\beta} \vec{m}^t L_s^{-1} \vec{g}_{a,b} \Big) \right|_{a,b=0} \\
&= \frac{1}{\beta} \vec{e}_k^{\,t} L_s^{-1} \vec{e}_l + \frac{s^2}{\beta^2} (\vec{m} L_s^{-1} \vec{e}_k) \cdot (\vec{m}^t L_s^{-1} \vec{e}_l), \\
\lim_{n \to \infty} \mathbb{E}_{n,\beta}^{[s]} [ X_k ] 
&= \left. \partial_a \exp \Big( \frac{a^2}{2\beta} \vec{e}_k^{\,t} L_s^{-1} \vec{e}_k + \frac{as}{\beta} \vec{m} L_s^{-1} \vec{e}_k) \Big) \right|_{a=0} \\
&= \frac{s}{\beta} \vec{m} L_s^{-1} \vec{e}_k.
\end{align*}

Therefore,
\begin{align*}
\lim_{n \to \infty} &\left(\frac{\partial_s Z_{n,\beta}^{[s]}}{Z_{n,\beta}^{[s]}}-M_n\right) \\
&= -\frac{1}{2} \sum_{k,l \ge 1} K_{kl} (L_s^{-1})_{kl} 
- \frac{s^2}{2\beta} \sum_{k,l \ge 1} K_{kl} \big(\vec{m} L_s^{-1} \big)_k \cdot \big(\vec{m} L_s^{-1} \big)_l + \frac{s}{\beta} \vec{m} L_s^{-1} \vec{m}^t \\
&= -\frac{1}{2} \operatorname{Tr} (L_s^{-1} K) - \frac{s^2}{2\beta} \vec{m} L_s^{-1} K L_s^{-1} \vec{m}^t +  \frac{s}{\beta} \vec{m} L_s^{-1} \vec{m}^t.
\end{align*}
By making use of the fact that
    $$ \partial_s \log\det L_s = \operatorname{Tr}(L_s^{-1} K),\quad \partial_s L_s^{-1} = -L_s^{-1} K L_s^{-1}, $$
we can write the right hand side as
    $$ \partial_s \left[ -\frac{1}{2} \log\det L_s + \frac{s^2}{2\beta} \vec{m} L_s^{-1} \vec{m}^t \right]. $$
As a consequence,
\begin{align*}
\lim_{n\to\infty} \log \frac{\bar{Z}_{n,\beta}^{\gamma,V}}{\bar{Z}_{n,\beta}^{\mathbb{T},0}}  &= -\frac{1}{2} \log\det L_1 + \frac{1}{2} \log\det L_0 + \frac{1}{2\beta} \vec{m} L_1^{-1} \vec{m}^t.
\end{align*}
By Proposition \ref{V_coeff} and Proposition \ref{V_DE}, we have
    $$ \frac{1}{2\beta} \vec{m} L_1^{-1} \vec{m}^t = \frac{1}{2}\int h_1 m^{(1)} d\mu_\tau =\frac{1}{\beta} \mathcal{D}_{\C\setminus\gamma}(m\circ z_e^{-1}), $$
where $h$ is the solution of the master equation \eqref{ME} given $g=m$. After an application of Proposition \ref{repr_eq_dens}, this finishes the proof.
\end{proof}

\section*{Acknowledgments}

The work of K.J. and T.W. was supported by the grant KAW 2023.0216 from the Knut and Alice Wallenberg Foundation. K.J. was also supported by a grant from the Swedish Research Council (VR). We thank Daniel Ofner for many helpful discussions about the equilibrium parametrization. We are also grateful to Fredrik Viklund for providing some comments that helped improve the presentation of Section 1--3.

\end{document}